\documentclass[11pt]{amsart}

\usepackage[T1]{fontenc}
\usepackage[utf8]{inputenc}
\usepackage{lmodern}
\usepackage[a4paper,margin=30mm]{geometry}
\usepackage{amssymb,mathtools,bm}
\usepackage{microtype}
\usepackage{booktabs}
\usepackage{tabularx}
\usepackage{graphicx}
\usepackage{xcolor}
\usepackage{enumitem}
\usepackage{float}
\usepackage[section]{placeins}
\usepackage[numbers,sort&compress]{natbib}
\usepackage[
  hidelinks,
  pdfencoding=auto,
  psdextra,
  pdftitle={Wasserstein Stability and Free Boundaries in Measure-Parameterized Bilevel Obstacle Problems},
  pdfauthor={Kun Huang},
  pdfsubject={Variational inequalities, bilevel optimization, and free-boundary stability under probability-law perturbations},
  pdfkeywords={obstacle problem, variational inequality, bilevel optimization, Wasserstein stability, free boundary, optimal taxation},
  pdflang={en-US},
  pdfdisplaydoctitle=true
]{hyperref}
\usepackage[nameinlink,capitalise,noabbrev]{cleveref}

\numberwithin{equation}{section}

\theoremstyle{plain}
\newtheorem{theorem}{Theorem}[section]
\newtheorem{proposition}[theorem]{Proposition}
\newtheorem{lemma}[theorem]{Lemma}
\newtheorem{corollary}[theorem]{Corollary}
\theoremstyle{definition}
\newtheorem{assumption}[theorem]{Assumption}
\newtheorem{example}[theorem]{Example}

\theoremstyle{remark}
\newtheorem{remark}[theorem]{Remark}

\AddToHook{env/theorem/begin}{\crefalias{section}{theorem}}
\AddToHook{env/proposition/begin}{\crefalias{theorem}{proposition}}
\AddToHook{env/lemma/begin}{\crefalias{theorem}{lemma}}
\AddToHook{env/corollary/begin}{\crefalias{theorem}{corollary}}
\AddToHook{env/assumption/begin}{\crefalias{theorem}{assumption}}
\AddToHook{env/example/begin}{\crefalias{theorem}{example}}
\AddToHook{env/definition/begin}{\crefalias{theorem}{definition}}
\AddToHook{env/remark/begin}{\crefalias{theorem}{remark}}
\crefname{assumption}{Assumption}{Assumptions}
\Crefname{assumption}{Assumption}{Assumptions}
\newcommand{\R}{\mathbb{R}}
\newcommand{\Pp}{\mathcal{P}}
\newcommand{\J}{\mathcal{J}}
\newcommand{\K}{\mathcal{K}}

\newcommand{\dd}{\,\mathrm{d}}
\newcommand{\Wass}{W_1}

\newcommand{\argmin}{\operatorname*{arg\,min}}
\newcommand{\argmax}{\operatorname*{arg\,max}}

\DeclareMathOperator{\Tr}{Tr}

\title[Wasserstein Stability of Bilevel Obstacles]{Wasserstein Stability and Free Boundaries in Measure-Parameterized Bilevel Obstacle Problems}
\author{Kun Huang}
\address{Economics and Management School, Wuhan University, Wuhan, China}
\email{huangkun123huang@163.com}
\urladdr{https://beibeihk.github.io/myblog/}
\date{September 5, 2026}

\begin{document}

\begin{abstract}
We study obstacle-constrained variational problems whose reduced energies depend on a probability law through a lower-level optimizer.
Uniform strong convexity yields a single-valued Lipschitz follower response, while convexity and a Poincar\'e inequality give a unique upper-level policy.
A type-Lipschitz reduced marginal then implies Lipschitz continuity of the policy map from the 1-Wasserstein metric to the energy space.
In a one-dimensional linear-obstacle subclass, the policy derivative is the positive part of a cumulative forcing.
Single crossing makes the coincidence set an interval.
An algebraic crossing of order \(m\) gives a \(W_1^{1/m}\) modulus for its endpoint; odd-power examples show that this exponent is sharp, while a transversal crossing recovers Lipschitz stability.
For empirical laws on compact subsets of \(\mathbb R^k\), dimension-dependent Wasserstein bounds yield finite-sample rates for policies and thresholds.
At a transversal population root, the empirical threshold is asymptotically linear, with an explicit influence function and central limit theorem.
For higher-dimensional regular patches, a conditional level-set argument gives local Hausdorff stability when a nondegenerate switching function is available.
An explicit quadratic firm response produces a nonlinear corporate-tax schedule with an endogenous zero-tax region and a Wasserstein-stable threshold.
The tax illustration is analytic and uses no empirical calibration.
\end{abstract}

\maketitle

\section{Introduction}
\label{sec:introduction}

Many optimization problems depend on a population only through its probability law.
The law may describe agents, firms, demand scenarios, or uncertain coefficients, while the decision variable is a function constrained by a pointwise inequality.
Two mathematical complications then interact.
First, the population enters after a lower-level optimization has been solved, so the upper objective is a reduced bilevel functional.
Second, an obstacle constraint partitions the domain into coincidence and positivity regions whose interface is endogenous.
Continuity of the optimizer does not by itself control that interface: a small vertical perturbation can move a nearly tangent zero by much more than the size of the perturbation.

This paper studies the complete perturbation chain
\[
\text{population law}
\longmapsto \text{follower response}
\longmapsto \text{obstacle policy}
\longmapsto \text{free boundary}.
\]
The metric on probability laws is the 1-Wasserstein distance.
Uniform strong convexity controls the follower response, and strong monotonicity controls the reduced variational inequality.
In a one-dimensional linear-obstacle subclass, integration of the Euler equation produces a scalar switching function whose nontrivial zero is the free boundary.
This reduction makes it possible to distinguish three levels of information: energy-space stability of the policy, a deterministic modulus for the boundary, and local statistical inference for the boundary point.

Variational inequalities (VIs) and obstacle problems have a classical existence, regularity, and free-boundary theory; foundational references include \citet{LionsStampacchia1967}, \citet{KinderlehrerStampacchia1980}, \citet{Friedman1982}, and \citet{Rodrigues1987}.
The functional-analytic tools used below are standard; useful general references are \citet{EkelandTemam1999} and \citet{Brezis2011}.
The regularity and geometric structure of the classical obstacle free boundary were developed in work of Caffarelli and many successors; see \citet{LewyStampacchia1969,Caffarelli1977,Caffarelli1998} and the monographs \citet{CaffarelliSalsa2005} and \citet{PetrosyanShahgholianUraltseva2012}.
More recently, \citet{FigalliSerra2019} refined the stratification and fine structure of the singular set, while \citet{FigalliRosOtonSerra2020} established generic regularity results.
Those works concern the fine geometry of a fixed obstacle problem; our question is the quantitative dependence of a selected interface on a probability-law parameter.

Sensitivity of variational inequalities has a separate history, including the differentiability analysis of \citet{Mignot1976} and the broader perturbation frameworks of \citet{BonnansShapiro2000}, \citet{RockafellarWets1998}, and \citet{DontchevRockafellar2009}.
Recent work goes substantially beyond classical directional differentiability.
For obstacle-type quasi-variational inequalities, \citet{ChristofWachsmuth2022} proved local Lipschitz stability and Hadamard directional differentiability of extremal solution maps.
For classical obstacle variational inequalities, \citet{ChristofWachsmuth2023,ChristofWachsmuth2025} obtained Newton differentiability in suitable Lebesgue-to-energy settings, while \citet{AlphonseWachsmuth2025} related derivatives of penalized problems to a Bouligand subdifferential.
For free boundaries specifically, \citet{Blank2001} established sharp regularity and stability results under Dini-type assumptions, and \citet{SerfatySerra2018} obtained quantitative motion laws for regular free boundaries under smooth perturbations of the obstacle.
At the level of contact-set measure, \citet{BlankLeCrone2019} proved linear estimates for symmetric differences under perturbations of the forcing and boundary data.
That topology differs from the pointwise displacement of a nondegenerate one-dimensional interface considered here.

The present analysis does not seek a new generalized derivative of the entire obstacle solution operator or a finer classification of singular free-boundary points.
Instead, it tracks a metric perturbation of a probability law through a reduced bilevel operator and then identifies the extra crossing condition needed to control the active-set boundary.
An algebraic crossing of order \(m\) yields a \(W_1^{1/m}\) boundary modulus.
The transversal case \(m=1\) gives Lipschitz motion and an influence-function expansion for empirical boundaries; a uniform odd-power family proves sharpness of the exponent when the crossing is degenerate.

The probability metric is natural because expectations of Lipschitz functions are stable in Wasserstein distance; see \citet{Villani2003,Villani2009,AmbrosioGigliSavare2008,Santambrogio2015}.
Measure-dependent stochastic programs are commonly analyzed through perturbations and sample-average approximations \citep{ShapiroDentchevaRuszczynski2014}.
Sharp empirical-Wasserstein behavior on compact metric spaces is governed by multiscale geometry and Wasserstein dimension \citep{WeedBach2019}; Euclidean moment bounds are developed by \citet{FournierGuillin2015}.
In one dimension, distribution-sensitive behavior can be slower without compactness or additional regularity \citep{BobkovLedoux2019}.
We retain the exact compact-interval bound obtained from the Dvoretzky--Kiefer--Wolfowitz inequality \citep{Massart1990}, and also state ambient-dimension-dependent guarantees on compact subsets of \(\mathbb R^k\).
These generic Wasserstein bounds are complemented by a root-\(N\) central limit theorem when the threshold is determined by a single transversal estimating equation.

Bilevel programming provides the conceptual setting for the lower and upper decisions \citep{Dempe2002,ColsonMarcotteSavard2007,DempeZemkoho2020}.
Modern algorithmic theory treats, among other regimes, smooth nonconvex--strongly-convex problems through implicit or iterative differentiation \citep{JiYangLiang2021} and possibly nonsmooth convex lower-level problems through first-order penalty methods \citep{LuMei2024}.
Our use of lower-level strong convexity is not an algorithmic complexity result.
It supplies a single-valued response whose dependence on the population law can be transferred to an infinite-dimensional obstacle variational inequality.

The application is a stylized nonlinear corporate-tax problem.
The classical optimal-tax literature supplies the foundations for production efficiency, behavioral responses to schedules, and heterogeneous taxpayers \citep{Mirrlees1971,DiamondMirrlees1971,Saez2001,DiamondSaez2011}.
For corporate taxation, heterogeneity, administrative costs, and differential treatment across firm sizes are central modeling considerations; examples include \citet{DevereuxGriffithKlemm2002}, \citet{DaviesEckel2010}, \citet{KeenMintz2004}, \citet{DharmapalaSlemrodWilson2011}, and \citet{BauerDaviesHaufler2014}.
The reduced treatment of multiple administrative and behavioral margins is also consistent with the tax-system perspective of \citet{SlemrodGillitzer2014}.
Our application is neither a general-equilibrium nor an empirical tax model.
It is an analytic verification case in which a firm chooses investment after observing a schedule, while the government trades off a local fiscal benefit, a type-dependent implementation cost, an activity benefit, and a quadratic smoothness cost.
The resulting schedule is zero below an endogenous firm-characteristic threshold and nonlinear above it.

The contributions are as follows.
\begin{enumerate}[label=(\roman*)]
\item We give primitive conditions under which the lower response is measurable, unique, Lipschitz in policy and type, and differentiable in policy.
\item We prove existence, uniqueness, complementarity, and a Wasserstein-Lipschitz estimate for a general convex reduced obstacle problem.
\item In one dimension, we derive an exact cumulative-forcing formula, prove interval structure of the contact set, and obtain \(C^1\) stability of the policy.
\item We prove a \(W_1^{1/m}\) free-boundary modulus under a uniform algebraic crossing of order \(m\).  The transversal theorem is the case \(m=1\), and a cubic family shows that the exponent \(1/3\) is optimal for that family.
\item We transfer compact Euclidean empirical-Wasserstein rates to policies and algebraically degenerate thresholds.  At a transversal root we further derive an asymptotic linear representation, an explicit influence function, and a central limit theorem.
\item For regular higher-dimensional spatial patches, we isolate the conditional level-set argument that turns a stable nondegenerate defining function into local Hausdorff stability.
\item In an explicit corporate-tax model, we verify the reduced assumptions and response properties directly, recover the general influence function in closed form, and show that the Lipschitz exponent for the threshold is attained by point-mass perturbations.
\end{enumerate}

These results form a unified deterministic-modulus, statistical-transfer, and local-limit chain for a measure-parameterized class of bilevel obstacle problems.
The scalar root perturbation, empirical-Wasserstein, and central-limit ingredients are standard in isolation; the contribution is their integration with the primitive-to-reduced variational analysis and the resulting separation between policy and interface stability.

The remainder is organized as follows.
\Cref{sec:setting} introduces the abstract spaces and assumptions.
\Cref{sec:response} treats the lower-level response and reduction, and \cref{sec:variational} proves the upper-level results.
\Cref{sec:one-dimensional,sec:fb-stability} develop the one-dimensional free-boundary theory and its deterministic moduli.
\Cref{sec:empirical} studies empirical laws and local threshold inference, while \cref{sec:higher-dimensional} gives the conditional regular-patch extension.
\Cref{sec:tax} verifies the corporate-tax application.
An analytic illustration in \cref{sec:numerical} concludes the application, followed by discussion, conclusions, and technical appendices.

\section{Setting and assumptions}
\label{sec:setting}

\subsection{Probability laws and the policy space}

Let \((\Theta,d_\Theta)\) be a Polish metric space and let \(\Pp_1(\Theta)\) denote the probability measures with finite first moment.
For \(\mu,\nu\in\Pp_1(\Theta)\),
\[
\Wass(\mu,\nu)
:=
\inf_{\pi\in\Pi(\mu,\nu)}
\int_{\Theta\times\Theta}d_\Theta(\theta,\eta)\,\pi(\dd\theta,\dd\eta),
\]
where \(\Pi(\mu,\nu)\) is the set of couplings.
Whenever \(h:\Theta\to\R\) is Lipschitz, the coupling definition immediately gives
\begin{equation}
\left|\int_\Theta h\,\dd\mu-\int_\Theta h\,\dd\nu\right|
\leq \operatorname{Lip}(h)\Wass(\mu,\nu).
\label{eq:KR}
\end{equation}
This is the only optimal-transport dual estimate required below.

Let \(\Omega\subset\R^d\) be bounded and Lipschitz.
Fix a relatively open boundary portion \(\Gamma_D\subset\partial\Omega\) of positive surface measure and define
\[
V:=\{v\in H^1(\Omega):\Tr v=0\text{ on }\Gamma_D\},
\qquad
\K:=\{v\in V:v\geq0\text{ a.e. in }\Omega\}.
\]
There is a constant \(C_P>0\) such that
\begin{equation}
\|v\|_{L^2(\Omega)}\leq C_P\|\nabla v\|_{L^2(\Omega)}
\quad\text{for every }v\in V.
\label{eq:poincare}
\end{equation}
We use \(\|v\|_V:=\|\nabla v\|_{L^2}\), an equivalent Hilbert norm on \(V\).
The cone \(\K\) is closed and convex.

\subsection{The bilevel model}

For \(x\in\Omega\), type \(\theta\in\Theta\), and scalar policy \(z\in\R\), a follower chooses \(y\in\R^m\) by solving
\begin{equation}
R(x,\theta,z)
:=\argmin_{y\in\R^m}\Phi(x,\theta,z,y).
\label{eq:response-definition}
\end{equation}
The upper-level primitive integrand is \(L(x,\theta,z,y)\).
After substituting the follower response, define
\begin{equation}
f(x,\theta,z):=L(x,\theta,z,R(x,\theta,z)),
\qquad
F_\mu(x,z):=\int_\Theta f(x,\theta,z)\,\mu(\dd\theta).
\label{eq:reduced-integrand}
\end{equation}
The policy problem is
\begin{equation}
(P_\mu)
\qquad
\min_{u\in\K}
\J_\mu(u),
\qquad
\J_\mu(u)
:=
\frac{\gamma}{2}\int_\Omega|\nabla u|^2\dd x
+\int_\Omega F_\mu(x,u(x))\dd x,
\label{eq:upper-problem}
\end{equation}
where \(\gamma>0\).

We separate primitive assumptions, which make \eqref{eq:response-definition} meaningful, from reduced assumptions, which drive the upper-level analysis.

\begin{assumption}[Uniformly strongly convex response]
\label{ass:response}
For almost every \(x\), the map \((\theta,z,y)\mapsto\Phi(x,\theta,z,y)\) is continuous, while \(x\mapsto\Phi(x,\theta,z,y)\) is measurable.
The function \(y\mapsto\Phi(x,\theta,z,y)\) is continuously differentiable and there is \(m_\Phi>0\) such that
\begin{equation}
\left\langle
\nabla_y\Phi(x,\theta,z,y_1)-\nabla_y\Phi(x,\theta,z,y_2),
y_1-y_2
\right\rangle
\geq m_\Phi|y_1-y_2|^2.
\label{eq:strong-convexity-response}
\end{equation}
There are constants \(L_z,L_\theta\geq0\) such that, for all admissible arguments,
\begin{align}
|\nabla_y\Phi(x,\theta,z_1,y)-\nabla_y\Phi(x,\theta,z_2,y)|
&\leq L_z|z_1-z_2|,
\label{eq:grad-lip-z}\\
|\nabla_y\Phi(x,\theta,z,y)-\nabla_y\Phi(x,\eta,z,y)|
&\leq L_\theta d_\Theta(\theta,\eta).
\label{eq:grad-lip-theta}
\end{align}
\end{assumption}

\begin{assumption}[Reduced convex integrand]
\label{ass:reduced}
For each \(\mu\) in a fixed class \(\mathfrak M\subset\Pp_1(\Theta)\), \(F_\mu\) is a Carath\'eodory integrand and is continuously differentiable in \(z\).
Write
\[
G_\mu(x,z):=\partial_zF_\mu(x,z).
\]
There exist \(a_0\in L^2(\Omega)\), \(a_1\geq0\), \(b_0\in L^1(\Omega)\), and \(\ell\in L^2(\Omega)\), independent of \(\mu\in\mathfrak M\), such that
\begin{align}
|F_\mu(x,0)|&\leq b_0(x),
\label{eq:F-zero}\\
|G_\mu(x,z)|&\leq a_0(x)+a_1|z|,
\label{eq:G-growth}\\
(G_\mu(x,z_1)-G_\mu(x,z_2))(z_1-z_2)&\geq0,
\label{eq:G-monotone}\\
|G_\mu(x,z)-G_\nu(x,z)|&\leq \ell(x)\Wass(\mu,\nu)
\label{eq:G-measure}
\end{align}
for almost every \(x\), all \(z,z_1,z_2\in\R\), and all \(\mu,\nu\in\mathfrak M\).
Moreover, there is \(b\in L^1(\Omega)\), independent of \(\mu\), such that
\begin{equation}
F_\mu(x,z)\geq-b(x)-c_0|z|
\label{eq:F-lower}
\end{equation}
for a constant \(c_0\geq0\).
\end{assumption}

Monotonicity in \eqref{eq:G-monotone} is equivalent to convexity of \(z\mapsto F_\mu(x,z)\).
No positive zero-order curvature is required: the gradient term and the Dirichlet trace already make \(\J_\mu\) strictly convex.
The uniform measure estimate \eqref{eq:G-measure} will follow from primitive type-Lipschitz estimates in \cref{sec:response}.

\subsection{Baseline and stronger assumptions}

\Cref{ass:response,ass:reduced} are the baseline hypotheses for the bilevel reduction, well-posedness, and energy-space stability.
The one-dimensional free-boundary theorem later adds continuity and monotonicity of a linear forcing.
An algebraic crossing supplies a Hölder modulus for the threshold, and transversality specializes it to a Lipschitz modulus and permits local statistical inference.
The higher-dimensional result assumes, rather than derives, a regular switching representation on a compact patch.
This hierarchy is intentional: solution stability is a monotonicity property, whereas free-boundary stability is a geometric property and requires more information.

\begin{remark}[Other Wasserstein orders]
If \(p\geq1\) and \(\mu,\nu\in\Pp_p(\Theta)\), then \(W_1(\mu,\nu)\leq W_p(\mu,\nu)\) by Jensen's inequality applied to every coupling.
Every bound below therefore remains valid with \(W_p\) on the right-hand side.
We work with \(W_1\) because Lipschitz type dependence gives its constants directly and because the one-dimensional empirical identity is simplest at order one.
\end{remark}

\subsection{Guide to the main results}

\Cref{tab:dependency} records the hierarchy of assumptions and conclusions used in the remainder of the paper.

\begin{table}[H]
\centering
\caption{Assumption and conclusion map.}
\label{tab:dependency}
\begin{tabularx}{\textwidth}{>{\raggedright\arraybackslash}p{0.27\textwidth}>{\raggedright\arraybackslash}p{0.30\textwidth}>{\raggedright\arraybackslash}X}
\toprule
Result & Main assumptions & Conclusion \\
\midrule
\Cref{thm:response} & uniform strong convexity and parameter-Lipschitz gradient & unique measurable Lipschitz response \\
\Cref{thm:wellposed} & reduced convexity, growth, Poincar\'e inequality & unique policy and variational inequality \\
\Cref{thm:H1-stability} & reduced marginal type-law estimate & \(H^1\) Wasserstein stability \\
\Cref{thm:exact-one-dimensional} & continuous single-crossing forcing and negative total & interval contact set and exact formula \\
\Cref{thm:algebraic-threshold-stability} & uniform algebraic crossing of order \(m_{\mathrm c}\) & \(W_1^{1/m_{\mathrm c}}\) free-boundary modulus \\
\Cref{thm:threshold-stability} & uniform switching transversality & Lipschitz free-boundary point \\
\Cref{thm:empirical-rate} & compact one-dimensional type space & expectation and high-probability rates \\
\Cref{thm:empirical-W1-dimension} & compact type space in \(\mathbb R^k\) & dimension-dependent policy and threshold rates \\
\Cref{thm:threshold-influence} & type-integral forcing and transversal population root & influence function and threshold CLT \\
\Cref{thm:higher-dimensional} & stable nondegenerate defining function & local Hausdorff stability \\
\Cref{thm:tax-schedule} & quadratic firm response and interior parameter class & explicit nonlinear tax schedule \\
\bottomrule
\end{tabularx}
\end{table}

\section{The lower-level response and the reduced integrand}
\label{sec:response}

We first justify that the response in \eqref{eq:response-definition} is a genuine function rather than a set-valued selection.

\begin{theorem}[Well-posed and stable follower response]
\label{thm:response}
Under \cref{ass:response}, for almost every \(x\) and every \((\theta,z)\), problem \eqref{eq:response-definition} has a unique solution.
The map \((x,\theta,z)\mapsto R(x,\theta,z)\) is jointly measurable after modification on an \(x\)-null set, and
\begin{align}
|R(x,\theta,z_1)-R(x,\theta,z_2)|
&\leq \frac{L_z}{m_\Phi}|z_1-z_2|,
\label{eq:R-lip-z}\\
|R(x,\theta,z)-R(x,\eta,z)|
&\leq \frac{L_\theta}{m_\Phi}d_\Theta(\theta,\eta).
\label{eq:R-lip-theta}
\end{align}
If, in addition, \(\Phi\) is twice continuously differentiable in \((z,y)\), then \(R\) is continuously differentiable in \(z\) and
\begin{equation}
D_zR(x,\theta,z)
=-
\bigl[D^2_{yy}\Phi(x,\theta,z,R(x,\theta,z))\bigr]^{-1}
D^2_{yz}\Phi(x,\theta,z,R(x,\theta,z)).
\label{eq:R-derivative}
\end{equation}
\end{theorem}

\begin{proof}
Fix \((x,\theta,z)\) outside the exceptional set in \(x\).
Strong monotonicity of the gradient implies strong convexity of \(y\mapsto\Phi(x,\theta,z,y)\).
Indeed, integration of the gradient along the segment from \(y_1\) to \(y_2\) gives
\[
\Phi(y_2)\geq
\Phi(y_1)+\langle\nabla_y\Phi(y_1),y_2-y_1\rangle
+\frac{m_\Phi}{2}|y_2-y_1|^2.
\]
Taking \(y_1=0\) shows coercivity in \(y_2\), hence existence of a minimizer in finite dimension.
Strict convexity makes it unique, and it is characterized by
\begin{equation}
\nabla_y\Phi(x,\theta,z,R(x,\theta,z))=0.
\label{eq:R-first-order}
\end{equation}

Let \(R_i=R(x,\theta,z_i)\).
By \eqref{eq:strong-convexity-response}, \eqref{eq:R-first-order}, and \eqref{eq:grad-lip-z},
\begin{align*}
m_\Phi|R_1-R_2|^2
&\leq
\langle\nabla_y\Phi(x,\theta,z_1,R_1)-\nabla_y\Phi(x,\theta,z_1,R_2),R_1-R_2\rangle\\
&=
\langle\nabla_y\Phi(x,\theta,z_2,R_2)-\nabla_y\Phi(x,\theta,z_1,R_2),R_1-R_2\rangle\\
&\leq L_z|z_1-z_2||R_1-R_2|.
\end{align*}
Cancellation proves \eqref{eq:R-lip-z}; the case \(R_1=R_2\) is immediate.
Replacing \eqref{eq:grad-lip-z} by \eqref{eq:grad-lip-theta} proves \eqref{eq:R-lip-theta}.

Continuity of the unique zero of the strongly monotone gradient follows directly from the two estimates and the assumed continuity of the primitive data.
For the \(x\)-variable, the Carath\'eodory property makes the graph of the unique minimizer measurable; equivalently, one may approximate the infimum over a countable dense subset of \(\R^m\) and use uniqueness.
Thus \(R\) is jointly measurable.
Finally, \eqref{eq:strong-convexity-response} implies that \(D^2_{yy}\Phi\) is invertible with inverse norm at most \(m_\Phi^{-1}\).
The implicit-function theorem applied to \eqref{eq:R-first-order} gives \eqref{eq:R-derivative}.
\end{proof}

The next statement supplies a direct route from primitive derivatives to the reduced measure estimate.

\begin{proposition}[A primitive route to the derivative measure estimate]
\label{prop:primitive-to-reduced}
Suppose the differentiable part of \cref{thm:response} holds, and suppose that \(L\) is continuously differentiable in \((z,y)\).
Define
\begin{equation}
g(x,\theta,z)
:=
\partial_zL(x,\theta,z,R(x,\theta,z))
+D_yL(x,\theta,z,R(x,\theta,z))D_zR(x,\theta,z).
\label{eq:primitive-g}
\end{equation}
For every \(\mu\in\mathfrak M\) and \(M>0\), assume that \(f(\cdot,\cdot,0)\) is integrable with respect to \(\dd x\otimes\mu\) and that there is an
\(e_{\mu,M}\in L^1(\Omega\times\Theta,\dd x\otimes\mu)\) such that
\begin{equation}
|g(x,\theta,z)|\leq e_{\mu,M}(x,\theta)
\quad\text{for almost every }(x,\theta)\text{ and every }|z|\leq M.
\label{eq:g-envelope}
\end{equation}
Assume also that \(z\mapsto f(x,\theta,z)\) is convex and
\begin{equation}
|g(x,\theta,z)-g(x,\eta,z)|
\leq \ell(x)d_\Theta(\theta,\eta)
\label{eq:g-type-lip}
\end{equation}
with \(\ell\in L^2(\Omega)\), uniformly in \(z\).
Then
\[
G_\mu(x,z)=\int_\Theta g(x,\theta,z)\,\mu(\dd\theta)
\]
is monotone in \(z\) and satisfies \eqref{eq:G-measure}.
\end{proposition}

\begin{proof}
The chain rule, \eqref{eq:g-envelope}, and dominated convergence give the displayed identity for \(G_\mu\), first on bounded \(z\)-intervals and hence for every \(z\in\R\).
Convexity of each \(f(x,\theta,\cdot)\) is preserved by integration, so its derivative \(G_\mu(x,\cdot)\) is monotone.
For any coupling \(\pi\in\Pi(\mu,\nu)\), \eqref{eq:g-type-lip} yields
\begin{align*}
|G_\mu(x,z)-G_\nu(x,z)|
&=
\left|
\int_{\Theta\times\Theta}
[g(x,\theta,z)-g(x,\eta,z)]\,\pi(\dd\theta,\dd\eta)
\right|\\
&\leq
\ell(x)
\int_{\Theta\times\Theta}d_\Theta(\theta,\eta)\,\pi(\dd\theta,\dd\eta).
\end{align*}
Taking the infimum over couplings proves \eqref{eq:G-measure}.
\end{proof}

\begin{remark}[Why the two levels are separated]
The response problem may be nonlinear even when the reduced upper integrand is simple.
Conversely, none of the arguments in \cref{sec:variational} needs to carry the variables \(y\) and \(\Phi\) once the properties of \(F_\mu\) have been verified.
This modularity is useful in applications: the lower model can change without altering the obstacle analysis.
\end{remark}

\section{The upper variational problem}
\label{sec:variational}

\subsection{Existence, uniqueness, and the variational inequality}

\begin{theorem}[Well-posedness of the optimal policy]
\label{thm:wellposed}
Under \cref{ass:reduced}, for every \(\mu\in\mathfrak M\) there exists a unique minimizer \(u_\mu\in\K\) of \eqref{eq:upper-problem}.
It is characterized by
\begin{equation}
\int_\Omega
\gamma\nabla u_\mu\cdot\nabla(v-u_\mu)
+G_\mu(x,u_\mu)(v-u_\mu)\dd x
\geq0
\quad\text{for every }v\in\K.
\label{eq:VI}
\end{equation}
\end{theorem}

\begin{proof}
By \eqref{eq:F-lower}, H\"older's inequality, and \eqref{eq:poincare},
\[
\J_\mu(u)
\geq
\frac{\gamma}{2}\|u\|_V^2
-c_0|\Omega|^{1/2}C_P\|u\|_V
-\|b\|_{L^1},
\]
so \(\J_\mu\) is coercive on \(\K\).
Let \((u_n)\subset\K\) be a minimizing sequence.
After taking a subsequence, \(u_n\rightharpoonup u\) in \(V\) and \(u_n\to u\) in \(L^2(\Omega)\).
The cone \(\K\) is weakly closed.
The gradient term is weakly lower semicontinuous.
In fact, the integral term is continuous along the strong \(L^2\) convergence just obtained: the mean-value theorem and \eqref{eq:G-growth} give an integrable product bound, while \eqref{eq:F-zero} makes every value finite.
Alternatively, convexity and the standard lower-semicontinuity theorem for integral functionals give
\[
\int_\Omega F_\mu(x,u(x))\dd x
\leq\liminf_{n\to\infty}
\int_\Omega F_\mu(x,u_n(x))\dd x.
\]
Thus \(u\) minimizes \(\J_\mu\).

The map \(u\mapsto\frac\gamma2\|u\|_V^2\) is strictly convex, while the integral term is convex.
Therefore \(\J_\mu\) is strictly convex on \(\K\), and the minimizer is unique.
For \(v\in\K\), the segment \(u_\mu+t(v-u_\mu)\) remains in \(\K\).
Taking the right derivative of \(\J_\mu\) at \(t=0\), justified by \eqref{eq:G-growth}, gives \eqref{eq:VI}.
Conversely, convexity implies
\[
\J_\mu(v)-\J_\mu(u_\mu)
\geq
\int_\Omega
\gamma\nabla u_\mu\cdot\nabla(v-u_\mu)
+G_\mu(x,u_\mu)(v-u_\mu)\dd x,
\]
so every solution of \eqref{eq:VI} is the minimizer.
\end{proof}

\subsection{Wasserstein stability of the policy}

\begin{theorem}[Energy-space stability]
\label{thm:H1-stability}
Let \(\mu,\nu\in\mathfrak M\).
Under \cref{ass:reduced},
\begin{align}
\|u_\mu-u_\nu\|_V
&\leq
\frac{C_P\|\ell\|_{L^2(\Omega)}}{\gamma}
\Wass(\mu,\nu),
\label{eq:V-stability}\\
\|u_\mu-u_\nu\|_{H^1(\Omega)}
&\leq
\frac{C_P\sqrt{1+C_P^2}\,\|\ell\|_{L^2(\Omega)}}{\gamma}
\Wass(\mu,\nu).
\label{eq:H1-stability}
\end{align}
If \(\Omega=(0,1)\) and \(\Gamma_D=\{0\}\), then also
\begin{equation}
\|u_\mu-u_\nu\|_{L^\infty(0,1)}
\leq
\frac{\|\ell\|_{L^2(0,1)}}{\gamma}
\Wass(\mu,\nu).
\label{eq:Linfty-stability}
\end{equation}
\end{theorem}

\begin{proof}
Use \(u_\nu\) as a test function in \eqref{eq:VI} for \(\mu\), and \(u_\mu\) as a test function in the corresponding inequality for \(\nu\).
With \(w=u_\mu-u_\nu\), addition gives
\[
\gamma\|w\|_V^2
+\int_\Omega
[G_\mu(x,u_\mu)-G_\nu(x,u_\nu)]w\dd x
\leq0.
\]
Insert and subtract \(G_\mu(x,u_\nu)\).
The term
\[
\int_\Omega
[G_\mu(x,u_\mu)-G_\mu(x,u_\nu)]w\dd x
\]
is nonnegative by \eqref{eq:G-monotone}.
Consequently, \eqref{eq:G-measure}, H\"older's inequality, and \eqref{eq:poincare} imply
\[
\gamma\|w\|_V^2
\leq
\|\ell\|_{L^2}\Wass(\mu,\nu)\|w\|_{L^2}
\leq
C_P\|\ell\|_{L^2}\Wass(\mu,\nu)\|w\|_V.
\]
This proves \eqref{eq:V-stability}; \eqref{eq:H1-stability} follows from Poincar\'e.
In one dimension, \(w(0)=0\) and the fundamental theorem for Sobolev functions gives
\[
|w(x)|\leq\int_0^1|w'(s)|\dd s\leq\|w'\|_{L^2(0,1)}.
\]
Since \(C_P\leq1\) is valid for the one-sided boundary condition on \((0,1)\), the same estimate yields \eqref{eq:Linfty-stability}.
\end{proof}

\begin{remark}
The proof uses no compactness in the measure variable and no differentiability of \(\mu\mapsto u_\mu\).
It is a strong-monotonicity estimate combined with the type-law bound \eqref{eq:G-measure}.
The constant is uniform over \(\mathfrak M\) because all reduced assumptions were imposed uniformly on that class.
\end{remark}

\subsection{A priori and value-function estimates}

The same monotonicity argument gives a uniform bound on the policies.

\begin{proposition}[Uniform policy bound]
\label{prop:uniform-policy-bound}
Under \cref{ass:reduced},
\begin{equation}
\|u_\mu\|_V
\leq\frac{C_P\|a_0\|_{L^2(\Omega)}}{\gamma}
\qquad(\mu\in\mathfrak M).
\label{eq:uniform-policy-bound}
\end{equation}
\end{proposition}

\begin{proof}
Use \(v=0\) in \eqref{eq:VI} and multiply the resulting inequality by \(-1\):
\[
\gamma\|u_\mu\|_V^2
+\int_\Omega G_\mu(x,u_\mu)u_\mu\dd x\leq0.
\]
Because \(u_\mu\geq0\), monotonicity implies
\[
[G_\mu(x,u_\mu)-G_\mu(x,0)]u_\mu\geq0.
\]
It follows from \eqref{eq:G-growth} and Poincar\'e that
\[
\gamma\|u_\mu\|_V^2
\leq-\int_\Omega G_\mu(x,0)u_\mu\dd x
\leq C_P\|a_0\|_{L^2}\|u_\mu\|_V.
\]
Cancellation proves the claim.
\end{proof}

Let \(m(\mu):=\min_{u\in\K}\J_\mu(u)\).
Policy stability does not automatically imply value stability unless the levels of the integrands are controlled as well as their derivatives.
The following explicit assumption supplies that control.

\begin{proposition}[Stability of the optimal value]
\label{prop:value-stability}
In addition to \cref{ass:reduced}, suppose there are \(r_0\in L^1(\Omega)\) and \(r_1\in L^2(\Omega)\) such that
\begin{equation}
|F_\mu(x,z)-F_\nu(x,z)|
\leq[r_0(x)+r_1(x)|z|]\Wass(\mu,\nu)
\label{eq:F-measure-level}
\end{equation}
for all \(z\) and almost every \(x\).
Then
\begin{equation}
|m(\mu)-m(\nu)|
\leq
\left(
\|r_0\|_{L^1}
+\frac{C_P^2\|r_1\|_{L^2}\|a_0\|_{L^2}}{\gamma}
\right)\Wass(\mu,\nu).
\label{eq:value-stability}
\end{equation}
\end{proposition}

\begin{proof}
Optimality of \(u_\mu\) gives
\begin{align*}
m(\mu)-m(\nu)
&=\J_\mu(u_\mu)-\J_\nu(u_\nu)\\
&\leq\J_\mu(u_\nu)-\J_\nu(u_\nu)\\
&=\int_\Omega[F_\mu(x,u_\nu)-F_\nu(x,u_\nu)]\dd x.
\end{align*}
By \eqref{eq:F-measure-level}, H\"older's inequality, \eqref{eq:poincare}, and \eqref{eq:uniform-policy-bound}, the last line is at most the right-hand side of \eqref{eq:value-stability}.
Interchanging \(\mu\) and \(\nu\) completes the proof.
\end{proof}

\subsection{Complementarity}

The variational inequality already gives a weak obstacle formulation.
When elliptic regularity makes the residual a function, it becomes the familiar complementarity system.

\begin{proposition}[Complementarity system]
\label{prop:complementarity}
Assume that \(u_\mu\in H^2(\Omega)\) and that \(G_\mu(\cdot,u_\mu)\in L^2(\Omega)\).
Set
\[
\xi_\mu:=-\gamma\Delta u_\mu+G_\mu(x,u_\mu).
\]
Then, almost everywhere in \(\Omega\),
\begin{equation}
u_\mu\geq0,
\qquad
\xi_\mu\geq0,
\qquad
u_\mu\xi_\mu=0.
\label{eq:complementarity}
\end{equation}
The prescribed trace is \(u_\mu=0\) on \(\Gamma_D\).
No pure Neumann condition is asserted on \(\partial\Omega\setminus\Gamma_D\): because nonnegativity also constrains boundary traces, the general natural condition there is of Signorini type.
Conversely, an \(H^2\) function in \(\K\) satisfying \eqref{eq:complementarity} and whose Green boundary term is nonnegative against every admissible trace difference solves \eqref{eq:VI}.
\end{proposition}

\begin{proof}
For every nonnegative \(\varphi\in C_c^\infty(\Omega)\), the competitor \(u_\mu+\varphi\) is admissible.
Equation \eqref{eq:VI} and integration by parts give \(\int\xi_\mu\varphi\geq0\), hence \(\xi_\mu\geq0\).
For \(0\leq\varphi\leq u_\mu\) with compact support, both \(u_\mu-\varphi\) and \(u_\mu+\varphi\) are admissible.
The two inequalities yield \(\int\xi_\mu\varphi=0\).
Approximation with \(\varphi=\min\{u_\mu,n\}\psi\), \(0\leq\psi\leq1\) and \(\psi\in C_c^\infty\), gives \(u_\mu\xi_\mu=0\).
Conversely, for \(v\in\K\),
\[
\int_\Omega\xi_\mu(v-u_\mu)\dd x
=\int_\Omega\xi_\mu v\dd x\geq0,
\]
and the assumed sign of the Green boundary term yields \eqref{eq:VI} after integration by parts.
\end{proof}

Classical obstacle regularity can now be invoked when its hypotheses hold.
For instance, with smooth domain, compatible boundary data, and bounded or H\"older forcing, standard results yield \(H^2\) or local \(C^{1,1}\) regularity; see \citet{KinderlehrerStampacchia1980,Rodrigues1987,PetrosyanShahgholianUraltseva2012}.
We do not require a multidimensional regularity theorem for the complete one-dimensional analysis below: its formula gives \(H^2\) and \(C^1\) regularity directly.

\section{One-dimensional obstacle structure}
\label{sec:one-dimensional}

We now specialize to an affine reduced integrand in one space dimension.
This is the minimal class in which the free boundary can be identified without importing a global regularity theory.
Let
\[
V=\{v\in H^1(0,1):v(0)=0\},
\qquad
\K=\{v\in V:v\geq0\text{ a.e.}\},
\]
and, for \(q_\mu\in C([0,1])\), consider
\begin{equation}
\J_\mu(u)
=\frac\gamma2\int_0^1|u'(x)|^2\dd x
-\int_0^1q_\mu(x)u(x)\dd x.
\label{eq:linear-energy}
\end{equation}
Thus \(F_\mu(x,z)=-q_\mu(x)z\) and \(G_\mu=-q_\mu\).
The natural boundary condition at \(x=1\) is \(u'(1)=0\).

\begin{assumption}[Single-crossing forcing]
\label{ass:single-crossing}
For every \(\mu\in\mathfrak M\), the forcing \(q_\mu\) is strictly increasing and there is \(c_\mu\in(0,1)\) such that
\begin{equation}
q_\mu(x)<0\quad(0\leq x<c_\mu),
\qquad
q_\mu(c_\mu)=0,
\qquad
q_\mu(x)>0\quad(c_\mu<x\leq1).
\label{eq:q-crossing}
\end{equation}
In addition,
\begin{equation}
\int_0^1q_\mu(s)\dd s<0.
\label{eq:q-negative-total}
\end{equation}
\end{assumption}

The negative total in \eqref{eq:q-negative-total} excludes a policy that is positive immediately to the right of the fixed boundary.
Positive forcing near \(x=1\) ensures a nonempty continuation region.
The assumptions are qualitative at this stage; their uniform quantitative versions enter in \cref{sec:fb-stability}.

Define the cumulative forcing
\begin{equation}
Q_\mu(x):=\int_x^1q_\mu(s)\dd s.
\label{eq:cumulative-forcing}
\end{equation}
Then \(Q_\mu(0)<0\), \(Q_\mu(c_\mu)>0\), \(Q_\mu(1)=0\), and \(Q_\mu'=-q_\mu\).

\begin{lemma}[The nontrivial switching point]
\label{lem:switching-point}
Under \cref{ass:single-crossing}, there is a unique \(b_\mu\in(0,c_\mu)\) such that
\begin{equation}
Q_\mu(b_\mu)=0.
\label{eq:threshold-equation}
\end{equation}
Moreover,
\begin{equation}
Q_\mu(x)<0\quad(0\leq x<b_\mu),
\qquad
Q_\mu(x)>0\quad(b_\mu<x<1).
\label{eq:Q-sign}
\end{equation}
\end{lemma}

\begin{proof}
On \([0,c_\mu]\), \(Q_\mu'=-q_\mu>0\) except at the right endpoint.
Since \(Q_\mu(0)<0<Q_\mu(c_\mu)\), there is exactly one zero \(b_\mu\in(0,c_\mu)\), and the asserted signs hold on \([0,c_\mu]\).
For \(c_\mu<x<1\),
\[
Q_\mu(x)=\int_x^1q_\mu(s)\dd s>0
\]
by \eqref{eq:q-crossing}.
\end{proof}

\begin{theorem}[Exact policy and interval contact set]
\label{thm:exact-one-dimensional}
Under \cref{ass:single-crossing}, the unique minimizer of \eqref{eq:linear-energy} is
\begin{equation}
u_\mu(x)
=\frac1\gamma
\int_0^x[Q_\mu(t)]_+\dd t
=
\begin{cases}
0, & 0\leq x\leq b_\mu,\\[0.3em]
\displaystyle\frac1\gamma\int_{b_\mu}^xQ_\mu(t)\dd t,
& b_\mu<x\leq1,
\end{cases}
\label{eq:exact-policy}
\end{equation}
where \([r]_+=\max\{r,0\}\).
It belongs to \(C^1([0,1])\cap H^2(0,1)\), satisfies \(u_\mu'(1)=0\), and has coincidence set and free boundary
\begin{equation}
\Lambda_\mu:=\{u_\mu=0\}=[0,b_\mu],
\qquad
\Gamma_\mu:=\partial\Lambda_\mu\cap(0,1)=\{b_\mu\}.
\label{eq:contact-free-boundary}
\end{equation}
The complementarity system is
\begin{equation}
u_\mu\geq0,
\qquad
-\gamma u_\mu''-q_\mu\geq0,
\qquad
u_\mu(-\gamma u_\mu''-q_\mu)=0
\quad\text{a.e. in }(0,1).
\label{eq:linear-complementarity}
\end{equation}
\end{theorem}

\begin{proof}
Define \(u_\mu\) by \eqref{eq:exact-policy}.
The sign description \eqref{eq:Q-sign} gives
\[
u_\mu'(x)=\gamma^{-1}[Q_\mu(x)]_+.
\]
Since \(Q_\mu\) is continuously differentiable and vanishes at \(b_\mu\) and \(1\), the derivative is continuous and \(u_\mu'(1)=0\).
The weak second derivative equals \(-q_\mu/\gamma\) on \((b_\mu,1)\) and zero on \((0,b_\mu)\), hence belongs to \(L^2(0,1)\).

On \((b_\mu,1)\), one has \(u_\mu>0\) and
\[
-\gamma u_\mu''-q_\mu=0.
\]
On \((0,b_\mu)\), \(u_\mu=0\) and \(q_\mu<0\) because \(b_\mu<c_\mu\), so
\[
-\gamma u_\mu''-q_\mu=-q_\mu>0.
\]
Thus \eqref{eq:linear-complementarity} holds.
For any \(v\in\K\), integration by parts and \(u_\mu'(1)=0\) give
\begin{align*}
\int_0^1\gamma u_\mu'(v-u_\mu)'-q_\mu(v-u_\mu)\dd x
&=\int_0^1(-\gamma u_\mu''-q_\mu)(v-u_\mu)\dd x\\
&=\int_0^{b_\mu}(-q_\mu)v\dd x\geq0.
\end{align*}
Hence \(u_\mu\) solves the variational inequality.
Coercivity and strict convexity of \eqref{eq:linear-energy} on \(\K\) make this minimizer unique.
The strict sign of \(Q_\mu\) to the right of \(b_\mu\) makes \(u_\mu\) strictly increasing there, proving \eqref{eq:contact-free-boundary}.
\end{proof}

\begin{corollary}[Quadratic nondegeneracy]
\label{cor:quadratic-growth}
If \(q_\mu\) is continuous, then
\begin{equation}
u_\mu(b_\mu+h)
=\frac{-q_\mu(b_\mu)}{2\gamma}h^2+o(h^2)
\qquad(h\downarrow0).
\label{eq:quadratic-growth}
\end{equation}
In particular, \(-q_\mu(b_\mu)>0\).
\end{corollary}

\begin{proof}
Because \(Q_\mu(b_\mu)=0\) and \(Q_\mu'(b_\mu)=-q_\mu(b_\mu)>0\),
\[
Q_\mu(b_\mu+t)=-q_\mu(b_\mu)t+o(t).
\]
Integrating from \(0\) to \(h\) in \eqref{eq:exact-policy} proves the result.
\end{proof}

The distinction between \eqref{eq:quadratic-growth} and threshold transversality is important.
The solution necessarily touches the obstacle with zero first derivative; the nondegenerate object is the switching function \(Q_\mu\), not \(u_\mu\) itself.

\subsection{A broader switching-function representation}

Strict monotonicity of \(q_\mu\) is a convenient primitive route to one interval, but the formula remains valid under a weaker sign condition.

\begin{proposition}[Sign-based representation]
\label{prop:sign-representation}
Let \(Q\in C^1([0,1])\) satisfy \(Q(1)=0\).
Suppose there is \(b\in(0,1)\) with \(Q<0\) on \([0,b)\), \(Q>0\) on \((b,1)\), and \(Q'\geq0\) on \([0,b]\).
Set \(q=-Q'\) and define \(u=\gamma^{-1}\int_0^x[Q(t)]_+\dd t\).
Then \(u\) is the unique minimizer of \(\frac\gamma2\int|u'|^2-\int qu\) over \(\K\), and its free boundary is \(\{b\}\).
\end{proposition}

\begin{proof}
The proof of \cref{thm:exact-one-dimensional} applies verbatim.
On the contact set, \(q=-Q'\leq0\), and on the positivity set \(-\gamma u''-q=0\).
\end{proof}

This slightly broader form will be used to construct the degenerate counterexample in \cref{sec:fb-stability}.

\section{Quantitative free-boundary stability}
\label{sec:fb-stability}

We first record that the one-dimensional policy enjoys a stronger estimate than the general energy-space bound.

\begin{proposition}[Uniform \(C^1\) stability]
\label{prop:C1-stability}
Let \(q_\mu,q_\nu\) satisfy \cref{ass:single-crossing}, with cumulative forcings \(Q_\mu,Q_\nu\).
Then
\begin{align}
\|u_\mu'-u_\nu'\|_{L^\infty(0,1)}
&\leq\frac1\gamma\|Q_\mu-Q_\nu\|_{L^\infty(0,1)},
\label{eq:derivative-stability}\\
\|u_\mu-u_\nu\|_{L^\infty(0,1)}
&\leq\frac1\gamma\|Q_\mu-Q_\nu\|_{L^\infty(0,1)}.
\label{eq:u-switch-stability}
\end{align}
If
\begin{equation}
\|q_\mu-q_\nu\|_{L^\infty(0,1)}
\leq L_q\Wass(\mu,\nu),
\label{eq:q-W1}
\end{equation}
then the right sides of \eqref{eq:derivative-stability} and \eqref{eq:u-switch-stability} are bounded by \(L_q\Wass(\mu,\nu)/\gamma\).
\end{proposition}

\begin{proof}
By \eqref{eq:exact-policy}, \(u_\mu'=\gamma^{-1}[Q_\mu]_+\).
The positive-part map is 1-Lipschitz, proving \eqref{eq:derivative-stability}.
Since \(u_\mu(0)=u_\nu(0)=0\), integration over an interval of length at most one proves \eqref{eq:u-switch-stability}.
Finally,
\[
|Q_\mu(x)-Q_\nu(x)|
\leq\int_x^1|q_\mu(s)-q_\nu(s)|\dd s
\leq L_q\Wass(\mu,\nu).
\]
\end{proof}

Policy stability alone does not locate a zero set.
We first record the square-root estimate furnished by state nondegeneracy, and then impose the stronger geometric condition directly on \(Q_\mu\).

\subsection{A square-root estimate from state nondegeneracy}

The next result formalizes the quadratic-growth mechanism associated with obstacle problems.

\begin{assumption}[Uniform quadratic growth]
\label{ass:uniform-quadratic}
There are \(c_*>0\) and \(\rho>0\) such that, for every \(\mu\in\mathfrak M\),
\begin{equation}
u_\mu(x)\geq c_*(x-b_\mu)^2
\quad\text{whenever }b_\mu\leq x\leq\min\{b_\mu+\rho,1\}.
\label{eq:uniform-quadratic}
\end{equation}
\end{assumption}

\begin{proposition}[Square-root boundary control]
\label{prop:sqrt-boundary}
Under \cref{ass:single-crossing,ass:uniform-quadratic}, let \(\mu,\nu\in\mathfrak M\) satisfy \(|b_\mu-b_\nu|\leq\rho\).
Then
\begin{equation}
|b_\mu-b_\nu|
\leq
\left(
\frac{\|u_\mu-u_\nu\|_{L^\infty(0,1)}}{c_*}
\right)^{1/2}.
\label{eq:sqrt-state-boundary}
\end{equation}
If \eqref{eq:q-W1} holds, then
\begin{equation}
|b_\mu-b_\nu|
\leq
\left(\frac{L_q}{\gamma c_*}\right)^{1/2}
\Wass(\mu,\nu)^{1/2}.
\label{eq:sqrt-W1-boundary}
\end{equation}
\end{proposition}

\begin{proof}
Assume \(b_\mu\leq b_\nu\); the other case is symmetric.
Because \(u_\nu(b_\nu)=0\), the uniform quadratic lower bound for \(u_\mu\) gives
\[
c_*(b_\nu-b_\mu)^2
\leq u_\mu(b_\nu)
=|u_\mu(b_\nu)-u_\nu(b_\nu)|
\leq\|u_\mu-u_\nu\|_\infty.
\]
Taking square roots proves \eqref{eq:sqrt-state-boundary}.
Equation \eqref{eq:sqrt-W1-boundary} follows from \cref{prop:C1-stability}.
\end{proof}

\begin{proposition}[A primitive route to uniform quadratic growth]
\label{prop:primitive-quadratic}
Suppose there are \(\kappa_*>0\) and \(\rho>0\) such that
\begin{equation}
-q_\mu(x)\geq\kappa_*
\quad\text{for }b_\mu\leq x\leq\min\{b_\mu+\rho,1\}
\label{eq:q-right-negative}
\end{equation}
for every \(\mu\in\mathfrak M\).
Then \cref{ass:uniform-quadratic} holds with \(c_*=\kappa_*/(2\gamma)\).
\end{proposition}

\begin{proof}
For \(0\leq t\leq\rho\) with \(b_\mu+t\leq1\),
\[
Q_\mu(b_\mu+t)
=\int_0^t[-q_\mu(b_\mu+s)]\dd s
\geq\kappa_*t.
\]
Integrating once more in \eqref{eq:exact-policy} gives
\[
u_\mu(b_\mu+t)
\geq\frac{\kappa_*}{\gamma}\int_0^t s\dd s
=\frac{\kappa_*}{2\gamma}t^2.
\]
\end{proof}

The square-root exponent follows from state stability plus quadratic growth.
It is not the best estimate when the cumulative switching equation can be compared directly.

\subsection{Lipschitz control from switching transversality}

\begin{assumption}[Uniform switching transversality]
\label{ass:transversality}
There are a compact interval \(I=[\underline b,\overline b]\subset(0,1)\) and \(\kappa>0\) such that \(b_\mu\in I\) for every \(\mu\in\mathfrak M\) and
\begin{equation}
Q_\mu'(x)=-q_\mu(x)\geq\kappa
\quad\text{for all }x\in I,
\quad \mu\in\mathfrak M.
\label{eq:uniform-Q-slope}
\end{equation}
\end{assumption}

\begin{theorem}[Lipschitz stability of the free boundary]
\label{thm:threshold-stability}
Suppose \cref{ass:single-crossing,ass:transversality} hold and \eqref{eq:q-W1} is valid uniformly over \(\mathfrak M\).
Then
\begin{equation}
|b_\mu-b_\nu|
\leq \frac{L_q}{\kappa}\Wass(\mu,\nu)
\qquad(\mu,\nu\in\mathfrak M).
\label{eq:threshold-W1}
\end{equation}
Equivalently, because each free boundary is a singleton,
\[
d_H(\Gamma_\mu,\Gamma_\nu)
\leq \frac{L_q}{\kappa}\Wass(\mu,\nu).
\]
\end{theorem}

\begin{proof}
Assume without loss of generality that \(b_\mu\geq b_\nu\).
Both points lie in \(I\), so \eqref{eq:uniform-Q-slope} and the fundamental theorem of calculus give
\begin{align*}
\kappa(b_\mu-b_\nu)
&\leq Q_\mu(b_\mu)-Q_\mu(b_\nu)\\
&=Q_\nu(b_\nu)-Q_\mu(b_\nu)\\
&\leq\|Q_\mu-Q_\nu\|_{L^\infty(0,1)}.
\end{align*}
The last quantity is at most \(L_q\Wass(\mu,\nu)\) by the final estimate in the proof of \cref{prop:C1-stability}.
\end{proof}

The theorem identifies the exponent \(\beta=1\) under switching transversality.
It is stronger than the square-root estimate one could infer by combining uniform solution stability with quadratic growth of \(u_\mu\).
That weaker route treats the graph of \(u_\mu\) but ignores the scalar balance equation \(Q_\mu(b_\mu)=0\).

\subsection{Algebraically degenerate crossings}

Transversality is the first member of a hierarchy of quantitative crossing conditions.
The following formulation uses only a lower growth bound for the switching function and therefore also covers nonsmooth crossings.

\begin{assumption}[Uniform algebraic crossing]
\label{ass:algebraic-crossing}
There are an integer \(m_{\mathrm c}\geq1\), a compact interval \(I\Subset(0,1)\), and a constant \(\kappa_{m_{\mathrm c}}>0\) with the following properties.
For every \(\mu\in\mathfrak M\), the switching function \(Q_\mu\) satisfies the hypotheses of \cref{prop:sign-representation}, its nontrivial zero \(b_\mu\) belongs to \(I\), and
\begin{equation}
(x-b_\mu)Q_\mu(x)
\geq
\kappa_{m_{\mathrm c}}|x-b_\mu|^{m_{\mathrm c}+1}
\qquad(x\in I).
\label{eq:algebraic-crossing}
\end{equation}
\end{assumption}

Equivalently, away from the zero,
\[
\operatorname{sgn}(x-b_\mu)Q_\mu(x)
\geq \kappa_{m_{\mathrm c}}|x-b_\mu|^{m_{\mathrm c}}.
\]
When this condition is obtained from a smooth Taylor expansion and the first nonzero derivative has order \(m_{\mathrm c}\), a sign-changing crossing forces \(m_{\mathrm c}\) to be odd.

\begin{theorem}[H\"older stability under an algebraic crossing]
\label{thm:algebraic-threshold-stability}
Suppose \cref{ass:algebraic-crossing} holds and there is \(L_Q\geq0\) such that
\begin{equation}
\|Q_\mu-Q_\nu\|_{L^\infty(I)}
\leq L_Q\Wass(\mu,\nu)
\qquad(\mu,\nu\in\mathfrak M).
\label{eq:switching-W1}
\end{equation}
Then
\begin{equation}
|b_\mu-b_\nu|
\leq
\left(\frac{L_Q}{\kappa_{m_{\mathrm c}}}\right)^{1/m_{\mathrm c}}
\Wass(\mu,\nu)^{1/m_{\mathrm c}}.
\label{eq:algebraic-threshold-W1}
\end{equation}
If \eqref{eq:q-W1} holds, then \eqref{eq:switching-W1} is valid with \(L_Q=L_q\).
\end{theorem}

\begin{proof}
Assume first that \(b_\mu\leq b_\nu\).
Applying \eqref{eq:algebraic-crossing} to \(Q_\mu\) at \(x=b_\nu\) and using \(Q_\nu(b_\nu)=0\) gives
\[
\kappa_{m_{\mathrm c}}(b_\nu-b_\mu)^{m_{\mathrm c}}
\leq Q_\mu(b_\nu)
=Q_\mu(b_\nu)-Q_\nu(b_\nu)
\leq\|Q_\mu-Q_\nu\|_{L^\infty(I)}.
\]
Interchanging \(\mu\) and \(\nu\) proves the same estimate when \(b_\nu<b_\mu\), and \eqref{eq:switching-W1} yields \eqref{eq:algebraic-threshold-W1}.
Finally, \eqref{eq:q-W1} implies
\[
|Q_\mu(x)-Q_\nu(x)|
\leq\int_x^1|q_\mu(s)-q_\nu(s)|\dd s
\leq L_q\Wass(\mu,\nu),
\]
which proves the last assertion.
\end{proof}

For \(m_{\mathrm c}=1\), \cref{thm:algebraic-threshold-stability} recovers the Lipschitz exponent in \cref{thm:threshold-stability}.
For higher crossing orders it gives the precise loss of regularity caused by degeneracy.

\subsection{Failure of Lipschitz control at a degenerate crossing}

The next family has uniformly small forcing perturbations but a much larger displacement of the free boundary.

\begin{example}[A cubic, degenerate crossing]
\label{ex:degenerate}
Fix \(b_0\in(0,1)\), choose \(a_*>0\) so that \(b_0+a_*^{1/3}<1\), and let \(a\in[0,a_*]\).
Define
\begin{equation}
Q_a(x):=(1-x)\bigl((x-b_0)^3-a\bigr),
\qquad
q_a(x):=-Q_a'(x).
\label{eq:degenerate-Q}
\end{equation}
The nontrivial zero is
\begin{equation}
b_a=b_0+a^{1/3}.
\label{eq:degenerate-threshold}
\end{equation}
Moreover, \(Q_a<0\) on \([0,b_a)\), \(Q_a>0\) on \((b_a,1)\), and
\[
Q_a'(x)
=a-(x-b_0)^3+3(1-x)(x-b_0)^2\geq0
\quad(0\leq x\leq b_a).
\]
Thus \cref{prop:sign-representation} applies.
At \(a=0\), however,
\[
Q_0'(b_0)=0.
\]
The perturbation of the forcing is exactly
\begin{equation}
\|q_a-q_0\|_{L^\infty(0,1)}=a,
\label{eq:forcing-a}
\end{equation}
whereas \(|b_a-b_0|=a^{1/3}\).
Consequently there is no constant \(C\) such that
\[
|b_a-b_0|\leq C\|q_a-q_0\|_{L^\infty}
\]
for all sufficiently small \(a\).

The example in fact satisfies the uniform cubic condition in \cref{ass:algebraic-crossing}.
To see this, choose a compact interval \(I\Subset(0,1)\) containing \(b_a\) for every \(a\in[0,a_*]\), and set \(\varepsilon_I:=1-\sup I>0\).
Writing \(y=x-b_0\) and \(d=a^{1/3}\), one has
\[
|Q_a(x)|
=(1-x)|y-d|(y^2+yd+d^2).
\]
The identity
\[
4(y^2+yd+d^2)-(y-d)^2=3(y+d)^2
\]
therefore gives
\begin{equation}
\operatorname{sgn}(x-b_a)Q_a(x)
\geq\frac{\varepsilon_I}{4}|x-b_a|^3
\qquad(x\in I).
\label{eq:cubic-uniform-crossing}
\end{equation}

To embed the same example in the measure-parameterized framework, take \(\Theta=[0,a_*]\), set
\[
\bar a_\mu:=\int_\Theta\theta\,\mu(\dd\theta),
\qquad
Q_\mu(x):=(1-x)\bigl((x-b_0)^3-\bar a_\mu\bigr),
\qquad
q_\mu:=-Q_\mu',
\]
and note that \(q_\mu=q_0-\bar a_\mu\).
Thus \(\|q_\mu-q_\nu\|_\infty=|\bar a_\mu-\bar a_\nu|\leq\Wass(\mu,\nu)\).
Now let \(\mu_a=\delta_a\) and \(\mu_0=\delta_0\); these laws recover \(Q_a\) and \(Q_0\).
Then \(\Wass(\mu_a,\mu_0)=a\), so
\[
|b_{\mu_a}-b_{\mu_0}|=\Wass(\mu_a,\mu_0)^{1/3}.
\]
Consequently the exponent \(1/3\) in
\cref{thm:algebraic-threshold-stability} is sharp on this family:
no estimate with a uniform constant and exponent \(\beta>1/3\) can hold near \(\mu_0\).
The policies still satisfy an \(O(a)\) \(C^1\) estimate.
Indeed, \cref{prop:sign-representation} gives \(u_a'=\gamma^{-1}[Q_a]_+\), while
\(\|Q_a-Q_0\|_\infty=\|a(1-\cdot)\|_\infty=a\); the positive-part contraction and integration therefore give
\[
\|u_a'-u_0'\|_\infty+\|u_a-u_0\|_\infty\leq\frac{2a}{\gamma}.
\]
It is only the zero-set location that loses the Lipschitz rate.
\end{example}

\begin{remark}
The exponent \(1/3\) in \cref{ex:degenerate} is not presented as a universal rate.
Replacing the cubic by an odd power \(m\geq3\) produces \(a^{1/m}\).
The example shows that some quantitative crossing information is indispensable; the rate is determined by the order of degeneracy.
\end{remark}

\subsection{A type-integral forcing}

The Wasserstein hypothesis \eqref{eq:q-W1} is automatic for the most common measure dependence.

\begin{proposition}[Measure-to-forcing estimate]
\label{prop:measure-forcing}
Let \(\varrho:[0,1]\times\Theta\to\R\) satisfy
\[
|\varrho(x,\theta)-\varrho(x,\eta)|\leq L_q d_\Theta(\theta,\eta)
\]
uniformly in \(x\), and set
\[
q_\mu(x)=\int_\Theta \varrho(x,\theta)\,\mu(\dd\theta).
\]
Then \eqref{eq:q-W1} holds.
\end{proposition}

\begin{proof}
Apply the coupling argument in \eqref{eq:KR} pointwise in \(x\), then take the supremum over \(x\).
\end{proof}

\subsection{Order comparative statics}

The cumulative formula also gives a comparison theorem that does not require differentiability with respect to a parameter.

\begin{theorem}[Monotone forcing comparative statics]
\label{thm:order-comparative-statics}
Let \(q_1,q_2\) satisfy \cref{ass:single-crossing}, and let \(u_1,u_2\) and \(b_1,b_2\) be their policies and thresholds.
If
\[
q_1(x)\leq q_2(x)
\quad\text{for every }x\in[0,1],
\]
then
\begin{equation}
u_1(x)\leq u_2(x)
\quad(0\leq x\leq1),
\qquad
b_1\geq b_2.
\label{eq:order-comparative}
\end{equation}
If \(q_1<q_2\) on a set of positive measure in \((b_2,1)\), then \(b_1>b_2\).
\end{theorem}

\begin{proof}
Let \(Q_i(x)=\int_x^1q_i(s)\dd s\).
The forcing order implies \(Q_1\leq Q_2\), hence \([Q_1]_+\leq[Q_2]_+\).
Integrating the exact derivative formula gives \(u_1\leq u_2\).

At \(b_2\), one has \(Q_2(b_2)=0\), so \(Q_1(b_2)\leq0\).
The sign description for \(Q_1\) says that it is negative to the left of \(b_1\) and positive to the right.
Therefore \(b_2\leq b_1\).
The final assertion follows by making the cumulative inequality at \(b_2\) strict.
\end{proof}

The direction in \eqref{eq:order-comparative} is natural for the minimization convention in \eqref{eq:linear-energy}: a larger marginal forcing makes a positive policy more attractive and weakly shrinks the contact interval.

\subsection{Differentiable motion along a parametric path}

For a smooth one-parameter family, the threshold has an explicit first variation.
This result complements the metric estimate in \cref{thm:threshold-stability}.

\begin{theorem}[Velocity of the one-dimensional free boundary]
\label{thm:threshold-velocity}
Let \(t\mapsto q_t\in C([0,1])\) be continuously differentiable in the uniform norm for \(t\) in an open interval.
Assume each \(q_t\) satisfies \cref{ass:single-crossing}, and suppose
\[
q_t(b_t)\neq0.
\]
Then \(t\mapsto b_t\) is continuously differentiable and
\begin{equation}
b_t'
=
\frac{\displaystyle\int_{b_t}^1\partial_tq_t(s)\dd s}
{q_t(b_t)}.
\label{eq:threshold-velocity}
\end{equation}
For a fixed \(x>b_t\) that remains in the positivity region for nearby parameters,
\begin{equation}
\partial_tu_t(x)
=\frac1\gamma\int_{b_t}^x
\left(\int_s^1\partial_tq_t(r)\dd r\right)\dd s.
\label{eq:policy-velocity}
\end{equation}
For a fixed \(x<b_t\) that stays in the contact region, \(\partial_tu_t(x)=0\).
\end{theorem}

\begin{proof}
Set
\[
\mathcal Q(t,x):=\int_x^1q_t(s)\dd s.
\]
Then \(\mathcal Q(t,b_t)=0\) and
\[
\partial_x\mathcal Q(t,b_t)=-q_t(b_t)\neq0.
\]
The implicit-function theorem gives differentiability of \(b_t\).
Differentiating the zero identity yields
\[
0=\int_{b_t}^1\partial_tq_t(s)\dd s-q_t(b_t)b_t',
\]
which proves \eqref{eq:threshold-velocity}.

For \(x>b_t\), differentiate
\[
u_t(x)=\gamma^{-1}\int_{b_t}^x\mathcal Q(t,s)\dd s.
\]
The moving-boundary term is \(-\gamma^{-1}\mathcal Q(t,b_t)b_t'=0\), and differentiation under the integral gives \eqref{eq:policy-velocity}.
The formula in the interior of the contact region is immediate because the solution vanishes there for nearby \(t\).
\end{proof}

\begin{remark}[Affine path]
For \(q_t(x)=\alpha x-h(t)\), formula \eqref{eq:threshold-velocity} reduces to
\[
b_t'=\frac{2h'(t)}\alpha,
\]
in agreement with the explicit threshold \(b_t=2h(t)/\alpha-1\).
\end{remark}

\section{Empirical measures and finite-sample thresholds}
\label{sec:empirical}

Let \(\Theta=[\theta_-,\theta_+]\subset\R\), write \(D=\theta_+-\theta_-\), and let \(\theta_1,\ldots,\theta_N\) be independent and identically distributed (i.i.d.) with law \(\mu\).
The empirical law is
\begin{equation}
\mu_N:=\frac1N\sum_{i=1}^N\delta_{\theta_i}.
\label{eq:empirical-law}
\end{equation}
We assume that \(\mathfrak M\) contains \(\mu\) and the empirical laws under consideration and that the constants in the preceding sections are uniform on this class.

\begin{theorem}[Deterministic empirical implication]
\label{thm:empirical-deterministic}
Under the hypotheses of \cref{thm:H1-stability},
\begin{equation}
\|u_{\mu_N}-u_\mu\|_{H^1(\Omega)}
\leq C_u\Wass(\mu_N,\mu),
\qquad
C_u:=\frac{C_P\sqrt{1+C_P^2}\,\|\ell\|_{L^2}}{\gamma}.
\label{eq:empirical-policy}
\end{equation}
Under the hypotheses of \cref{thm:threshold-stability},
\begin{equation}
|b_{\mu_N}-b_\mu|
\leq C_b\Wass(\mu_N,\mu),
\qquad
C_b:=\frac{L_q}{\kappa}.
\label{eq:empirical-threshold}
\end{equation}
Consequently, \(u_{\mu_N}\to u_\mu\) in \(H^1\) and \(b_{\mu_N}\to b_\mu\) almost surely.
\end{theorem}

\begin{proof}
The deterministic inequalities are direct applications of \cref{thm:H1-stability,thm:threshold-stability}.
On a compact interval, the empirical measures converge weakly to \(\mu\) almost surely by the Glivenko--Cantelli theorem.
Compact support makes first moments uniformly integrable, hence weak convergence is equivalent to convergence in \(W_1\).
The almost-sure conclusions follow.
\end{proof}

We next make the rate explicit without imposing a density or a lower bound on it.
Let \(F\) and \(F_N\) denote the distribution functions of \(\mu\) and \(\mu_N\).
In one dimension,
\begin{equation}
\Wass(\mu_N,\mu)
=\int_{\theta_-}^{\theta_+}|F_N(t)-F(t)|\dd t
\leq D\|F_N-F\|_{L^\infty(\R)}.
\label{eq:W1-CDF}
\end{equation}

\begin{theorem}[Distribution-free finite-sample bounds]
\label{thm:empirical-rate}
For every probability law \(\mu\) on \([\theta_-,\theta_+]\),
\begin{equation}
\mathbb E\Wass(\mu_N,\mu)
\leq\frac{D}{2\sqrt N}.
\label{eq:expected-W1}
\end{equation}
For every \(\delta\in(0,1)\), with probability at least \(1-\delta\),
\begin{equation}
\Wass(\mu_N,\mu)
\leq
D\sqrt{\frac{\log(2/\delta)}{2N}}.
\label{eq:high-prob-W1}
\end{equation}
Therefore, under \cref{thm:empirical-deterministic},
\begin{align}
\mathbb E\|u_{\mu_N}-u_\mu\|_{H^1}
&\leq \frac{C_uD}{2\sqrt N},
\label{eq:expected-policy}\\
\mathbb E|b_{\mu_N}-b_\mu|
&\leq \frac{C_bD}{2\sqrt N},
\label{eq:expected-threshold}
\end{align}
and the corresponding high-probability bounds are obtained by multiplying the right side of \eqref{eq:high-prob-W1} by \(C_u\) or \(C_b\).
\end{theorem}

\begin{proof}
For fixed \(t\), \(NF_N(t)\) is binomial with success probability \(F(t)\).
Hence Jensen's inequality gives
\[
\mathbb E|F_N(t)-F(t)|
\leq\sqrt{\operatorname{Var}(F_N(t))}
=\sqrt{\frac{F(t)(1-F(t))}{N}}
\leq\frac1{2\sqrt N}.
\]
Integrating this estimate in the identity in \eqref{eq:W1-CDF} and using Tonelli proves \eqref{eq:expected-W1}.

The Dvoretzky--Kiefer--Wolfowitz--Massart inequality states
\[
\mathbb P\bigl(\|F_N-F\|_\infty>t\bigr)
\leq2e^{-2Nt^2}
\qquad(t>0);
\]
see \citet{Massart1990}.
Set \(t=\sqrt{\log(2/\delta)/(2N)}\) and use \eqref{eq:W1-CDF} to obtain \eqref{eq:high-prob-W1}.
The policy and threshold estimates then follow from \eqref{eq:empirical-policy} and \eqref{eq:empirical-threshold}.
\end{proof}

\begin{remark}[Dimension dependence]
The \(N^{-1/2}\) conclusion above is stated only for a compact one-dimensional type space.
In higher-dimensional type spaces, empirical Wasserstein rates depend on the dimension, the Wasserstein order, and available moments; see \citet{FournierGuillin2015}.
The next result records the dimension-dependent bound needed for the deterministic stability estimates.
\end{remark}

\subsection{Compact type spaces in arbitrary dimension}

For this subsection, let \(\Theta\subset\R^k\) be compact, where \(k\) is the dimension of the type space and is unrelated to the spatial dimension of \(\Omega\).
Write \(D=\operatorname{diam}(\Theta)>0\) and, for \(N\geq2\), set
\begin{equation}
\rho_k(N):=
\begin{cases}
N^{-1/2}, & k=1,\\
N^{-1/2}\log(1+N), & k=2,\\
N^{-1/k}, & k\geq3.
\end{cases}
\label{eq:dimension-rate}
\end{equation}

\begin{theorem}[Dimension-dependent empirical guarantees]
\label{thm:empirical-W1-dimension}
There is a constant \(A_k>0\), depending only on \(k\), such that
\begin{equation}
\mathbb E\Wass(\mu_N,\mu)
\leq A_kD\rho_k(N).
\label{eq:expected-W1-dimension}
\end{equation}
For \(\delta\in(0,1)\), define
\begin{equation}
R_{k,N,\delta}
:=A_kD\rho_k(N)
+D\sqrt{\frac{\log(1/\delta)}{2N}}.
\label{eq:dimension-high-prob-radius}
\end{equation}
Then, with probability at least \(1-\delta\),
\begin{equation}
\Wass(\mu_N,\mu)\leq R_{k,N,\delta}.
\label{eq:high-prob-W1-dimension}
\end{equation}
Under the hypotheses of \cref{thm:H1-stability}, these estimates imply
\begin{align}
\mathbb E\|u_{\mu_N}-u_\mu\|_{H^1(\Omega)}
&\leq C_uA_kD\rho_k(N),
\label{eq:expected-policy-dimension}\\
\|u_{\mu_N}-u_\mu\|_{H^1(\Omega)}
&\leq C_uR_{k,N,\delta}
\quad\text{with probability at least }1-\delta.
\label{eq:high-prob-policy-dimension}
\end{align}
If, in addition, \cref{ass:algebraic-crossing} and \eqref{eq:switching-W1}
hold for \(\mu\) and the empirical laws, then
\begin{align}
\mathbb E|b_{\mu_N}-b_\mu|
&\leq
\left(\frac{L_Q}{\kappa_{m_{\mathrm c}}}\right)^{1/m_{\mathrm c}}
\bigl(A_kD\rho_k(N)\bigr)^{1/m_{\mathrm c}},
\label{eq:expected-threshold-dimension}\\
|b_{\mu_N}-b_\mu|
&\leq
\left(\frac{L_Q}{\kappa_{m_{\mathrm c}}}\right)^{1/m_{\mathrm c}}
R_{k,N,\delta}^{1/m_{\mathrm c}}
\quad\text{with probability at least }1-\delta.
\label{eq:high-prob-threshold-dimension}
\end{align}
In particular, the transversal case \(m_{\mathrm c}=1\) inherits the empirical \(W_1\) rate, whereas an order-\(m_{\mathrm c}\) crossing inherits its \(1/m_{\mathrm c}\) power.
\end{theorem}

\begin{proof}
After translation and rescaling, it is enough to work in a unit cube.
Let \(\mathcal P_j\) be its dyadic partition into \(2^{jk}\) cubes at level \(j\).
A standard multiscale coupling bound gives, for every integer \(J\geq1\) (see, for example, \citet{WeedBach2019}),
\begin{equation}
\Wass(\mu_N,\mu)
\leq C_kD\left[
2^{-J}
+\sum_{j=1}^J2^{-j}
\sum_{C\in\mathcal P_j}|\mu_N(C)-\mu(C)|
\right],
\label{eq:dyadic-W1-bound}
\end{equation}
where \(C_k\) depends only on the dimension.
For each cell, \(N\mu_N(C)\) is binomial, and Cauchy--Schwarz yields
\[
\mathbb E\sum_{C\in\mathcal P_j}|\mu_N(C)-\mu(C)|
\leq
\frac1{\sqrt N}\sum_{C\in\mathcal P_j}\sqrt{\mu(C)}
\leq\frac{2^{jk/2}}{\sqrt N}.
\]
Taking expectations in \eqref{eq:dyadic-W1-bound} and choosing \(2^J\) of order \(N^{1/2}\) for \(k\leq2\) and of order \(N^{1/k}\) for \(k\geq3\) proves \eqref{eq:expected-W1-dimension}, after changing the dimension-only constant.

If one observation is replaced, the two empirical laws can be coupled with transportation cost at most \(D/N\).
The triangle inequality therefore shows that the value of \(\Wass(\mu_N,\mu)\) changes by at most \(D/N\).
McDiarmid's inequality gives
\[
\mathbb P\left(
\Wass(\mu_N,\mu)-\mathbb E\Wass(\mu_N,\mu)>t
\right)
\leq\exp\left(-\frac{2Nt^2}{D^2}\right).
\]
Taking \(t=D\sqrt{\log(1/\delta)/(2N)}\) proves
\eqref{eq:high-prob-W1-dimension}.
The policy estimates follow from \eqref{eq:empirical-policy}.
The high-probability threshold estimate follows from
\eqref{eq:algebraic-threshold-W1}; its expectation version follows from the concavity of \(s\mapsto s^{1/m_{\mathrm c}}\) and Jensen's inequality.
\end{proof}

The rate in \cref{thm:empirical-W1-dimension} is a distribution-free bound in the ambient type dimension.
Measures supported on a lower-dimensional set, and models depending on a finite-dimensional sufficient statistic, can converge faster.

\begin{remark}[Uniform admissibility]
In a general model, an empirical law may leave a structural class defined only by a strict condition on the population mean.
The corporate-tax application avoids this issue: its pointwise range restriction \(h(\theta)\in[h_-,h_+]\Subset(\alpha/2,\alpha)\) forces every empirical mean to remain in the same interval.
Thus its empirical bounds are unconditional within the stated type class.
\end{remark}

\subsection{Asymptotic inference for a type-integral threshold}

The generic Wasserstein route controls every Lipschitz dependence on the law, but it can be conservative when the free boundary is the root of a smooth scalar estimating equation.
We now derive its asymptotic linearization directly.
Here \((\Theta,d_\Theta)\) may be any Polish type space, and \(\mu_N\) again denotes the empirical law of an i.i.d. sample from \(\mu\).
Let
\begin{equation}
\begin{aligned}
q_\lambda(x)
&=\int_\Theta\varrho(x,\theta)\,\lambda(\dd\theta),\\
\psi_x(\theta)
&:=\int_x^1\varrho(s,\theta)\dd s,\\
Q_\lambda(x)
&=\int_\Theta\psi_x(\theta)\,\lambda(\dd\theta).
\end{aligned}
\label{eq:type-integral-switching}
\end{equation}

\begin{theorem}[Influence function and central limit theorem for the threshold]
\label{thm:threshold-influence}
Let \(\theta_1,\theta_2,\ldots\) be i.i.d. with law \(\mu\).
Suppose \(x\mapsto\varrho(x,\theta)\) is continuous for \(\mu\)-almost every \(\theta\), and there are measurable nonnegative functions \(M,L\) such that
\begin{align}
\sup_{x\in[0,1]}|\varrho(x,\theta)|&\leq M(\theta),
\label{eq:forcing-envelope}\\
|\varrho(x,\theta)-\varrho(y,\theta)|&\leq L(\theta)|x-y|,
\label{eq:forcing-random-Lipschitz}
\end{align}
with \(\int_\Theta(M^2+L)\dd\mu<\infty\).
Assume that \(Q_\mu(b_\mu)=0\) for some \(b_\mu\in(0,1)\), and that there are \(\rho,\kappa>0\) such that
\begin{equation}
[b_\mu-\rho,b_\mu+\rho]\Subset(0,1),
\qquad
q_\mu(x)\leq-\kappa
\quad\text{when }|x-b_\mu|\leq\rho.
\label{eq:statistical-transversality}
\end{equation}
Then, almost surely for every sufficiently large \(N\), \(Q_{\mu_N}\) has exactly one zero \(\widehat b_N\) in \([b_\mu-\rho,b_\mu+\rho]\), and \(\widehat b_N\to b_\mu\).
Moreover,
\begin{equation}
\widehat b_N-b_\mu
=\frac1N\sum_{i=1}^N\operatorname{IF}(\theta_i;\mu)
+o_{\mathbb P}(N^{-1/2}),
\label{eq:threshold-asymptotic-linearization}
\end{equation}
where
\begin{equation}
\operatorname{IF}(\theta;\mu)
:=\frac{\psi_{b_\mu}(\theta)}{q_\mu(b_\mu)}
=\frac{\displaystyle\int_{b_\mu}^1\varrho(s,\theta)\dd s}
{q_\mu(b_\mu)}.
\label{eq:threshold-influence-function}
\end{equation}
Consequently,
\begin{equation}
\sqrt N(\widehat b_N-b_\mu)
\xrightarrow{\ d\ }
\mathcal N(0,\sigma_b^2),
\qquad
\sigma_b^2
:=\frac{\operatorname{Var}_\mu(\psi_{b_\mu}(\theta))}
{q_\mu(b_\mu)^2}.
\label{eq:threshold-general-CLT}
\end{equation}
If the empirical switching functions satisfy the sign representation in \cref{prop:sign-representation} and their nontrivial zeros lie in the displayed neighborhood, then \(\widehat b_N=b_{\mu_N}\) is the empirical free-boundary point.
\end{theorem}

\begin{proof}
The envelope and random Lipschitz conditions imply the uniform laws of large numbers
\begin{equation}
\sup_{x\in[0,1]}|q_{\mu_N}(x)-q_\mu(x)|\longrightarrow0,
\qquad
\sup_{x\in[0,1]}|Q_{\mu_N}(x)-Q_\mu(x)|\longrightarrow0
\quad\text{a.s.}
\label{eq:forcing-uniform-LLN}
\end{equation}
Indeed, one applies the scalar strong law on a finite mesh and controls the gaps by the empirical averages of \(L\) for \(q\), and of \(M\) for \(Q\).
Condition \eqref{eq:statistical-transversality} gives
\(Q_\mu(b_\mu-\rho)\leq-\kappa\rho\) and
\(Q_\mu(b_\mu+\rho)\geq\kappa\rho\).
By \eqref{eq:forcing-uniform-LLN}, eventually these endpoint signs persist and \(q_{\mu_N}\leq-\kappa/2\) throughout the interval.
Thus \(Q_{\mu_N}'=-q_{\mu_N}\geq\kappa/2\) there, proving existence and uniqueness of \(\widehat b_N\).
The same uniform convergence, together with strict monotonicity, gives \(\widehat b_N\to b_\mu\).

For some random point \(\xi_N\) between \(b_\mu\) and \(\widehat b_N\), the mean-value theorem gives
\[
0=Q_{\mu_N}(\widehat b_N)
=Q_{\mu_N}(b_\mu)-q_{\mu_N}(\xi_N)(\widehat b_N-b_\mu).
\]
Because \(Q_\mu(b_\mu)=0\),
\[
Q_{\mu_N}(b_\mu)
=\frac1N\sum_{i=1}^N\psi_{b_\mu}(\theta_i),
\]
and \eqref{eq:forcing-uniform-LLN} and consistency imply
\(q_{\mu_N}(\xi_N)\to q_\mu(b_\mu)\) in probability.
The numerator is \(O_{\mathbb P}(N^{-1/2})\), because \(|\psi_{b_\mu}|\leq M\in L^2(\mu)\).
Replacing the random denominator by \(q_\mu(b_\mu)\) therefore has remainder \(o_{\mathbb P}(N^{-1/2})\), which proves \eqref{eq:threshold-asymptotic-linearization}.
The classical central limit theorem proves \eqref{eq:threshold-general-CLT}.
\end{proof}

For a finite signed measure \(H\) with \(H(\Theta)=0\), suppose that
\(\mu+tH\) remains a probability measure and its local root remains transversal for sufficiently small \(t\geq0\).
Differentiation of \(Q_{\mu+tH}(b_{\mu+tH})=0\) at \(t=0\) gives
\begin{equation}
D b_\mu[H]
=\frac{\displaystyle\int_\Theta\psi_{b_\mu}(\theta)H(\dd\theta)}
{q_\mu(b_\mu)}.
\label{eq:threshold-functional-derivative}
\end{equation}
Taking \(H=\delta_\theta-\mu\) recovers
\eqref{eq:threshold-influence-function} because
\(\int\psi_{b_\mu}\dd\mu=Q_\mu(b_\mu)=0\).

\begin{remark}[Transversal and degenerate statistical regimes]
The denominator in \eqref{eq:threshold-influence-function} is nonzero by
\eqref{eq:statistical-transversality}.
At an algebraically degenerate crossing it vanishes, and a root-\(N\) Gaussian limit should not be expected.
Formally, if
\(Q_\mu(b_\mu+h)=ch^{m_{\mathrm c}}+o(h^{m_{\mathrm c}})\) with \(m_{\mathrm c}>1\) and the empirical fluctuation at \(b_\mu\) is of order \(N^{-1/2}\), then the root moves on the slower scale \(N^{-1/(2m_{\mathrm c})}\) and generally has a nonlinear, non-Gaussian limit.
Thus \cref{thm:threshold-influence} and \cref{thm:empirical-W1-dimension} describe complementary structural and worst-case routes.
\end{remark}

\section{Higher-dimensional regular patches}
\label{sec:higher-dimensional}

In dimensions \(d\geq2\), an arbitrary obstacle free boundary can contain both regular and singular points, and its global geometry need not be stable in a strong norm.
We therefore make no global regularity claim.
This section gives the parameter-dependent step that applies after classical regularity theory or a problem-specific transform has produced a nondegenerate local defining function.

Let \(U\Subset\Omega\) be open.
For each \(\mu\in\mathfrak M\), suppose a relatively closed regular patch \(\Gamma_\mu^U\subset U\) is represented by
\begin{equation}
\Gamma_\mu^U=\{x\in U:\Psi_\mu(x)=0\},
\label{eq:defining-function}
\end{equation}
where \(\Psi_\mu\in C^1(U)\).
The function \(\Psi_\mu\) is a switching or hodograph variable, not necessarily the obstacle solution itself.
Indeed, the solution normally has zero gradient at a regular contact point.

\begin{assumption}[Uniform regular patch]
\label{ass:regular-patch}
There exist \(U_0\Subset U\), \(\rho>0\), \(\kappa>0\), and \(L_\Psi>0\) such that for all \(\mu,\nu\in\mathfrak M\):
\begin{enumerate}[label=(\alph*)]
\item \(\Gamma_\mu^U\subset U_0\), and the closed \(\rho\)-neighborhood of \(U_0\) is contained in \(U\);
\item \(|\nabla\Psi_\mu(x)|\geq\kappa\) throughout the closed \(\rho\)-neighborhood of \(U_0\);
\item \(\|\Psi_\mu-\Psi_\nu\|_{L^\infty(U)}\leq L_\Psi\Wass(\mu,\nu)\).
\end{enumerate}
\end{assumption}

\begin{lemma}[Stability of nondegenerate level sets]
\label{lem:level-set}
Let \(\Psi,\widetilde\Psi\in C^1(U)\) have zero sets contained in \(U_0\), and suppose both gradients have norm at least \(\kappa\) throughout the closed \(\rho\)-neighborhood of \(U_0\).
If
\[
\varepsilon:=\|\Psi-\widetilde\Psi\|_{L^\infty(U)}<\kappa\rho,
\]
then
\begin{equation}
d_H(\{\Psi=0\},\{\widetilde\Psi=0\})
\leq\frac{\varepsilon}{\kappa}.
\label{eq:level-set-Hausdorff}
\end{equation}
\end{lemma}

\begin{proof}
Fix \(x\) with \(\Psi(x)=0\).
Then \(|\widetilde\Psi(x)|\leq\varepsilon\).
Consider the local flow \(X(t)\) solving
\[
X'(t)=\frac{\nabla\widetilde\Psi(X(t))}{|\nabla\widetilde\Psi(X(t))|^2},
\qquad X(0)=x.
\]
As long as the flow stays in the \(\rho\)-neighborhood of \(U_0\),
\[
\frac{\dd}{\dd t}\widetilde\Psi(X(t))=1,
\qquad
|X'(t)|\leq\kappa^{-1}.
\]
Run the flow for time \(-\widetilde\Psi(x)\), reversing its orientation if necessary.
The traveled distance is at most \(\varepsilon/\kappa<\rho\), and the endpoint \(y\) satisfies \(\widetilde\Psi(y)=0\).
Thus \(\operatorname{dist}(x,\{\widetilde\Psi=0\})\leq\varepsilon/\kappa\).
Interchanging the functions proves the reverse directed distance and hence \eqref{eq:level-set-Hausdorff}.
Standard local existence for the flow suffices; a smooth approximation gives the same conclusion when the defining functions are merely \(C^1\).
\end{proof}

\begin{theorem}[Local Hausdorff stability of regular patches]
\label{thm:higher-dimensional}
Under \cref{ass:regular-patch}, if
\[
\Wass(\mu,\nu)<\frac{\kappa\rho}{L_\Psi},
\]
then
\begin{equation}
d_H(\Gamma_\mu^U,\Gamma_\nu^U)
\leq\frac{L_\Psi}{\kappa}\Wass(\mu,\nu).
\label{eq:higher-Hausdorff}
\end{equation}
\end{theorem}

\begin{proof}
Apply \cref{lem:level-set} with \(\varepsilon=L_\Psi\Wass(\mu,\nu)\).
\end{proof}

\begin{remark}[Content and limitation of the theorem]
The parameter-dependent conclusion in \eqref{eq:higher-Hausdorff} follows from the stable defining-function assumptions.
Obtaining such a function from a particular multidimensional obstacle problem may require the classical regularity theory and additional smooth dependence on the data.
Theorems such as those of \citet{Blank2001} and \citet{SerfatySerra2018} provide substantially deeper information in their respective settings.
\Cref{thm:higher-dimensional} is best read as a reusable final step for measure-parametric models, not as a replacement for that theory.
\end{remark}

\section{Application to a nonlinear corporate-tax schedule}
\label{sec:tax}

This section gives a fully explicit application of the abstract theory.
The model uses a reduced set of primitives.
Its purpose is to show how a heterogeneous lower-level behavioral response can generate the measure-dependent forcing used above, not to replace a structural public-finance model.

\subsection{Firm response}

The observable state \(x\in(0,1)\) is a normalized firm characteristic, such as taxable-profit rank or a size index.
The government chooses a schedule \(\tau(x)\).
Fix a statutory ceiling \(\bar\tau\in(0,1)\) and define
\[
\K_{\bar\tau}
:=
\{\tau\in H^1(0,1):\tau(0)=0,\ 0\leq\tau\leq\bar\tau\text{ a.e.}\}.
\]
The endpoint normalization \(\tau(0)=0\) anchors the schedule at the bottom of the normalized characteristic range.

A firm of type \(\theta\in\Theta\) chooses investment \(k\geq0\) after observing the local rate \(z=\tau(x)\):
\begin{equation}
k^*(\theta,z)
=\argmax_{k\geq0}
\left\{(1-z)A(\theta)k-\frac{r(\theta)}2k^2\right\}.
\label{eq:firm-problem}
\end{equation}
Assume \(A,r:\Theta\to(0,\infty)\) are bounded and Lipschitz, with \(r(\theta)\geq r_0>0\).
Equivalently, the firm minimizes
\[
\Phi(\theta,z,k)
=\frac{r(\theta)}2k^2-(1-z)A(\theta)k.
\]
Since \(0\leq z\leq\bar\tau<1\), the unconstrained optimizer is positive and
\begin{equation}
k^*(\theta,z)=\frac{A(\theta)}{r(\theta)}(1-z).
\label{eq:firm-response}
\end{equation}
The response is endogenous and affine in the local scalar rate; the optimal schedule derived below is nonlinear in \(x\).

\begin{proposition}[Direct verification of the firm response]
\label{prop:tax-response}
The firm problem has a unique response, and that response obeys
\begin{align}
|k^*(\theta,z_1)-k^*(\theta,z_2)|
&\leq \left\|\frac Ar\right\|_\infty|z_1-z_2|,
\label{eq:tax-response-z}\\
|k^*(\theta,z)-k^*(\eta,z)|
&\leq \operatorname{Lip}(A/r)d_\Theta(\theta,\eta).
\label{eq:tax-response-theta}
\end{align}
\end{proposition}

\begin{proof}
The second derivative of \(\Phi\) in \(k\) is \(r(\theta)\geq r_0\).
Formula \eqref{eq:firm-response} follows from the first-order condition and positivity.
The two estimates follow directly, using \(0<1-z\leq1\) in the second.
\end{proof}

When \(r\) varies with type, the primitive gradient \(r(\theta)k-(1-z)A(\theta)\) need not satisfy the globally uniform type estimate in \eqref{eq:grad-lip-theta} for every \(k\in\R\).
The explicit calculation above supplies exactly the response properties needed here, so the application does not invoke that stronger sufficient assumption.

\subsection{Government loss and reduction}

Let \(\alpha>0\) be the slope of the local fiscal value of a unit of tax liability, let \(c(\theta)\geq0\) be a type-dependent implementation or distortion coefficient, and let \(\omega>0\) be the weight on firm activity.
For a local policy \(z\) and response \(k\), take the primitive government loss
\begin{equation}
L(x,\theta,z,k)
=-\alpha xz+c(\theta)z-\omega k.
\label{eq:government-primitive}
\end{equation}
The first term is the local fiscal benefit, the second is an implementation or distortion cost, and the last values investment.
The total loss also includes the smoothness cost \(\frac\gamma2\int_0^1|\tau'|^2\), which represents the administrative and behavioral cost of abrupt differentiation across nearby observable firm states.
Both \(x\) and the term \(\alpha xz\) are reduced-form cardinal objects; in particular, \eqref{eq:government-primitive} is not a structural accounting identity for statutory tax revenue.

Substitution of \eqref{eq:firm-response} into \eqref{eq:government-primitive} gives
\begin{align}
f(x,\theta,z)
&=-\omega\frac{A(\theta)}{r(\theta)}
+\left[c(\theta)+\omega\frac{A(\theta)}{r(\theta)}-\alpha x\right]z\\
&=-\omega\frac{A(\theta)}{r(\theta)}-q(x,\theta)z,
\label{eq:tax-reduced-integrand}
\end{align}
where
\begin{equation}
h(\theta):=c(\theta)+\omega\frac{A(\theta)}{r(\theta)},
\qquad
q(x,\theta):=\alpha x-h(\theta).
\label{eq:h-q-tax}
\end{equation}
For \(\theta\sim\mu\), write
\begin{equation}
\bar h_\mu:=\int_\Theta h(\theta)\,\mu(\dd\theta),
\qquad
q_\mu(x)=\alpha x-\bar h_\mu.
\label{eq:tax-forcing}
\end{equation}
Up to the policy-independent constant \(-\omega\int A/r\,\dd\mu\), the government problem is
\begin{equation}
\min_{\tau\in\K_{\bar\tau}}
\left\{
\frac\gamma2\int_0^1|\tau'(x)|^2\dd x
-\int_0^1(\alpha x-\bar h_\mu)\tau(x)\dd x
\right\}.
\label{eq:tax-upper-problem}
\end{equation}

Under \cref{ass:tax-class} below, extend the affine expression in \eqref{eq:tax-reduced-integrand} from the admissible interval \([0,\bar\tau]\) to all \(z\in\R\); this does not alter the constrained policy problem.
The resulting reduced integrand satisfies \cref{ass:reduced} directly.
Indeed, its derivative is \(G_\mu(x,z)=-q_\mu(x)=\bar h_\mu-\alpha x\), which is monotone in \(z\), and \eqref{eq:KR} gives
\begin{equation}
|G_\mu(x,z)-G_\nu(x,z)|
\leq \operatorname{Lip}(h)\Wass(\mu,\nu).
\label{eq:tax-reduced-W1}
\end{equation}
The remaining growth and lower bounds are uniform because \(A/r\) and \(h\) are bounded.

We now state the parameter restrictions transparently.

\begin{assumption}[Interior tax-threshold class]
\label{ass:tax-class}
The function \(h\) is bounded and Lipschitz, and its range is contained in
\[
[h_-,h_+]\Subset(\alpha/2,\alpha).
\]
Finally, the statutory ceiling satisfies
\begin{equation}
\frac{\alpha}{12\gamma}
\left(2-\frac{2h_-}{\alpha}\right)^3
<\bar\tau<1.
\label{eq:upper-obstacle-inactive}
\end{equation}
\end{assumption}

The exact formulas below do not require an additional common-interval condition.
If one instead wants to invoke the generic transversality theorem verbatim, the optional restriction
\begin{equation}
2h_+-h_-<\alpha
\label{eq:tax-uniform-crossing}
\end{equation}
verifies \cref{ass:transversality} with
\[
I=\left[\frac{2h_-}{\alpha}-1,\frac{2h_+}{\alpha}-1\right],
\qquad
\kappa=\alpha+h_--2h_+>0.
\]

The lower bound on \(\bar\tau\) guarantees that the analytically derived schedule does not touch the upper obstacle.
If the ceiling is lower, the model becomes a double-obstacle problem with a second free boundary, which is outside the present theorem chain.

\begin{theorem}[Explicit optimal corporate-tax schedule]
\label{thm:tax-schedule}
Under \cref{ass:tax-class}, problem \eqref{eq:tax-upper-problem} has a unique solution.
Define
\begin{equation}
b_\mu:=\frac{2\bar h_\mu}{\alpha}-1\in(0,1).
\label{eq:tax-threshold}
\end{equation}
Then
\begin{equation}
\tau_\mu(x)
=
\begin{cases}
0, & 0\leq x\leq b_\mu,\\[0.4em]
\displaystyle
\frac{\alpha}{12\gamma}(x-b_\mu)^2
\bigl[3(1-b_\mu)-2(x-b_\mu)\bigr],
& b_\mu<x\leq1.
\end{cases}
\label{eq:explicit-tax-schedule}
\end{equation}
The schedule is \(C^1\), strictly increasing on \((b_\mu,1)\), and satisfies
\begin{equation}
\tau_\mu(1)=\frac{\alpha}{12\gamma}(1-b_\mu)^3<\bar\tau.
\label{eq:tax-maximum}
\end{equation}
Thus the zero-tax region is \([0,b_\mu]\), the upper obstacle is inactive, and the only interior free-boundary point is \(b_\mu\).
\end{theorem}

\begin{proof}
The affine forcing \(q_\mu(x)=\alpha x-\bar h_\mu\) is strictly increasing, crosses zero at \(c_\mu=\bar h_\mu/\alpha\in(1/2,1)\), and has
\[
\int_0^1q_\mu(x)\dd x=\frac\alpha2-\bar h_\mu<0.
\]
Therefore \cref{ass:single-crossing} holds.
The cumulative forcing is
\begin{align}
Q_\mu(x)
&=\frac\alpha2(1-x^2)-\bar h_\mu(1-x)\\
&=\frac\alpha2(1-x)(x-b_\mu),
\label{eq:tax-Q}
\end{align}
where the second equality follows from \eqref{eq:tax-threshold}.
Its nontrivial zero is \(b_\mu\), which lies in \((0,c_\mu)\).
Integrating \(Q_\mu/\gamma\) from \(b_\mu\) to \(x\) yields \eqref{eq:explicit-tax-schedule}.
The derivative equals \(Q_\mu/\gamma>0\) on \((b_\mu,1)\), and evaluation at \(x=1\) gives \eqref{eq:tax-maximum}.
Because \(b_\mu\geq2h_-/\alpha-1\), condition \eqref{eq:upper-obstacle-inactive} implies \(\tau_\mu(1)<\bar\tau\).
The lower-obstacle minimizer is therefore feasible for \(\K_{\bar\tau}\) and remains the unique minimizer on this smaller set.
\end{proof}

\begin{proposition}[Shape and value of the tax schedule]
\label{prop:tax-shape-value}
Under \cref{ass:tax-class}, the positive part of the schedule has one inflection point at
\begin{equation}
x_\mu^{\mathrm{inf}}=\frac{1+b_\mu}{2}.
\label{eq:tax-inflection}
\end{equation}
More precisely,
\begin{equation}
\tau_\mu'(x)=\frac\alpha{2\gamma}(1-x)(x-b_\mu),
\qquad
\tau_\mu''(x)=\frac\alpha{2\gamma}(1+b_\mu-2x)
\label{eq:tax-derivatives}
\end{equation}
for \(b_\mu<x<1\).
If \(m_{\mathrm{red}}(\mu)\) denotes the minimum of \eqref{eq:tax-upper-problem} after omitting its policy-independent constant, then
\begin{equation}
m_{\mathrm{red}}(\mu)
=-\frac{\alpha^2}{240\gamma}(1-b_\mu)^5
=-\frac{2\alpha^2}{15\gamma}
\left(1-\frac{\bar h_\mu}{\alpha}\right)^5.
\label{eq:tax-optimal-value}
\end{equation}
Consequently, along a differentiable path \(t\mapsto\bar h_t\),
\begin{equation}
\frac{\dd}{\dd t}m_{\mathrm{red}}(t)
=\frac{2\alpha}{3\gamma}
\left(1-\frac{\bar h_t}{\alpha}\right)^4\bar h_t'.
\label{eq:tax-value-derivative}
\end{equation}
\end{proposition}

\begin{proof}
Differentiate \eqref{eq:explicit-tax-schedule}, or use \(\tau_\mu'=Q_\mu/\gamma\) and \eqref{eq:tax-Q}, to obtain \eqref{eq:tax-derivatives}.
The second derivative changes sign once, at \eqref{eq:tax-inflection}.

Complementarity and integration by parts give the energy identity
\[
\gamma\int_0^1|\tau_\mu'|^2\dd x
=\int_0^1q_\mu\tau_\mu\dd x.
\]
Hence the reduced minimum is \(-\frac\gamma2\int|\tau_\mu'|^2\).
With \(L=1-b_\mu\) and \(s=x-b_\mu\),
\begin{align*}
\int_0^1|\tau_\mu'|^2\dd x
&=\frac{\alpha^2}{4\gamma^2}
\int_0^L s^2(L-s)^2\dd s\\
&=\frac{\alpha^2L^5}{120\gamma^2}.
\end{align*}
This proves the first equality in \eqref{eq:tax-optimal-value}; the second follows from \(1-b_\mu=2(1-\bar h_\mu/\alpha)\).
Differentiation gives \eqref{eq:tax-value-derivative}.
\end{proof}

\begin{corollary}[Wasserstein stability of the tax schedule and threshold]
\label{cor:tax-stability}
Under \cref{ass:tax-class}, for any \(\mu,\nu\in\Pp_1(\Theta)\),
\begin{align}
|b_\mu-b_\nu|
&=\frac2\alpha|\bar h_\mu-\bar h_\nu|
\leq\frac{2\operatorname{Lip}(h)}{\alpha}\Wass(\mu,\nu),
\label{eq:tax-threshold-stability}\\
\|\tau_\mu'-\tau_\nu'\|_{L^\infty(0,1)}
&\leq\frac{\operatorname{Lip}(h)}{\gamma}\Wass(\mu,\nu),
\label{eq:tax-C1-derivative}\\
\|\tau_\mu-\tau_\nu\|_{L^\infty(0,1)}
&\leq\frac{\operatorname{Lip}(h)}{\gamma}\Wass(\mu,\nu).
\label{eq:tax-C1-level}
\end{align}
The threshold estimate is global on this class and does not require the less sharp generic constant \(L_q/\kappa\).
\end{corollary}

\begin{proof}
Formula \eqref{eq:tax-threshold} gives the equality.
The bound on \(\bar h_\mu-\bar h_\nu\) follows from \eqref{eq:KR}.
Since \(q_\mu-q_\nu=-(\bar h_\mu-\bar h_\nu)\), \cref{prop:C1-stability} gives the remaining estimates.
\end{proof}

\begin{corollary}[Empirical corporate-tax threshold]
\label{cor:tax-empirical}
Let \(\Theta=[\theta_-,\theta_+]\), let \(\mu_N\) be defined by \eqref{eq:empirical-law}, and let \(D=\theta_+-\theta_-\).
Then
\begin{equation}
\mathbb E|b_{\mu_N}-b_\mu|
\leq\frac{\operatorname{Lip}(h)D}{\alpha\sqrt N}.
\label{eq:tax-expected-threshold}
\end{equation}
With probability at least \(1-\delta\),
\begin{equation}
|b_{\mu_N}-b_\mu|
\leq
\frac{2\operatorname{Lip}(h)D}{\alpha}
\sqrt{\frac{\log(2/\delta)}{2N}}.
\label{eq:tax-high-prob-threshold}
\end{equation}
\end{corollary}

\begin{proof}
Combine \eqref{eq:tax-threshold-stability} with \cref{thm:empirical-rate}.
\end{proof}

The affine dependence on \(\bar h_\mu\) also permits bounds that use the range of \(h\) directly and avoid passing through the diameter of \(\Theta\).

\begin{theorem}[Direct finite-sample and limit theory for the tax threshold]
\label{thm:tax-sample-mean}
Let \(H=h_+-h_-\), let \(\sigma_h^2=\operatorname{Var}_\mu(h(\theta))\), and let \(\mu_N\) be the empirical law of an i.i.d. sample.
Then
\begin{equation}
\mathbb E|b_{\mu_N}-b_\mu|
\leq\frac{\sigma_h}{\alpha}\frac{2}{\sqrt N}
\leq\frac{H}{\alpha\sqrt N}.
\label{eq:tax-direct-mean}
\end{equation}
For every \(\delta\in(0,1)\), with probability at least \(1-\delta\),
\begin{equation}
|b_{\mu_N}-b_\mu|
\leq\frac{2H}{\alpha}
\sqrt{\frac{\log(2/\delta)}{2N}}.
\label{eq:tax-hoeffding}
\end{equation}
If \(\sigma_h^2>0\), then
\begin{equation}
\sqrt N\,(b_{\mu_N}-b_\mu)
\ \xrightarrow{\ d\ }\ 
\mathcal N\left(0,\frac{4\sigma_h^2}{\alpha^2}\right).
\label{eq:tax-clt}
\end{equation}
\end{theorem}

\begin{proof}
By \eqref{eq:tax-threshold},
\[
b_{\mu_N}-b_\mu
=\frac2\alpha
\left[\frac1N\sum_{i=1}^Nh(\theta_i)-\mathbb E_\mu h(\theta)\right].
\]
Jensen's inequality bounds the expected absolute sample-mean deviation by \(\sigma_h/\sqrt N\).
Popoviciu's variance inequality gives \(\sigma_h\leq H/2\), proving \eqref{eq:tax-direct-mean}.
Hoeffding's inequality for variables in an interval of length \(H\) gives
\[
\mathbb P\left(
\left|\frac1N\sum_{i=1}^Nh(\theta_i)-\mathbb E h(\theta)\right|>t
\right)
\leq2\exp\left(-\frac{2Nt^2}{H^2}\right).
\]
Set \(t=H\sqrt{\log(2/\delta)/(2N)}\) and multiply by \(2/\alpha\) to obtain \eqref{eq:tax-hoeffding}.

To identify \eqref{eq:tax-clt} as an exact instance of
\cref{thm:threshold-influence}, take
\(\varrho(x,\theta)=\alpha x-h(\theta)\) in
\eqref{eq:type-integral-switching}.
The boundedness of \(h\) verifies the envelope condition, and
\(x\mapsto\varrho(x,\theta)\) is Lipschitz with constant \(\alpha\).
Using \(Q_\mu(b_\mu)=0\) gives
\[
\psi_{b_\mu}(\theta)
=-(1-b_\mu)\bigl(h(\theta)-\bar h_\mu\bigr),
\qquad
q_\mu(b_\mu)=\bar h_\mu-\alpha.
\]
Since \(1-b_\mu=2(\alpha-\bar h_\mu)/\alpha\), the general influence function reduces to
\begin{equation}
\operatorname{IF}_{\mathrm{tax}}(\theta;\mu)
=\frac2\alpha\bigl(h(\theta)-\bar h_\mu\bigr).
\label{eq:tax-influence-function}
\end{equation}
Every empirical mean remains in \([h_-,h_+]\), so the empirical switching functions retain the sign representation and their local roots are precisely \(b_{\mu_N}\).
In fact, the exact identity in the first display of the proof says that the asymptotic linearization \eqref{eq:threshold-asymptotic-linearization} has zero remainder here.
Its influence-function variance is \(4\sigma_h^2/\alpha^2\), and \cref{thm:threshold-influence} proves \eqref{eq:tax-clt}.
\end{proof}

\begin{remark}[Two empirical routes]
\Cref{cor:tax-empirical} is inherited from the general Wasserstein theorem and remains meaningful when the model depends on the full law through many Lipschitz moments.
\Cref{thm:threshold-influence} gives a root-\(N\) limit for a transversal scalar estimating equation even when the ambient type dimension is high, whereas \cref{thm:empirical-W1-dimension} controls arbitrary Lipschitz law dependence and therefore carries the ambient-dimensional rate.
\Cref{thm:tax-sample-mean} is sharper still at finite samples because this affine example depends on \(\mu\) only through the scalar mean \(\bar h_\mu\).
These distinctions prevent an application-specific sufficient statistic from being mistaken for a dimension-free property of general measure-parameterized obstacle problems.
\end{remark}

\begin{proposition}[Attainment of the Lipschitz exponent]
\label{prop:tax-sharpness}
Suppose \(h(\theta)=\theta\) on a compact interval contained in \((\alpha/2,\alpha)\), and let \(\mu=\delta_\theta\), \(\nu=\delta_\eta\).
Then
\begin{equation}
|b_\mu-b_\nu|
=\frac2\alpha\Wass(\mu,\nu).
\label{eq:tax-sharp}
\end{equation}
Consequently the exponent one in \eqref{eq:tax-threshold-stability} cannot be improved to any exponent greater than one with a uniform constant near a fixed point mass.
\end{proposition}

\begin{proof}
For point masses, \(\Wass(\delta_\theta,\delta_\eta)=|\theta-\eta|\).
Substitution into \eqref{eq:tax-threshold} proves \eqref{eq:tax-sharp}.
If an estimate with exponent \(\beta>1\) held uniformly, division by \(|\theta-\eta|\) and passage to \(\eta\to\theta\) would give \(2/\alpha\leq0\), a contradiction.
\end{proof}

\subsection{Economic interpretation and boundaries}

An increase in \(\bar h_\mu\) shifts \(b_\mu\) to the right.
In the model, this composite term rises when the expected implementation cost is higher or when the activity benefit lost through taxation is larger.
The government then leaves a wider range of low-index firms at the lower obstacle.
Above the threshold, the schedule rises smoothly rather than jumping, because the gradient penalty assigns a cost to abrupt differentiation.

These conclusions are comparative statics of the stated reduced objective.
They are not empirical estimates and do not imply that an observed statutory schedule should have this shape.
The objective omits general-equilibrium incidence, avoidance across jurisdictions, and incentive-compatibility constraints associated with private information.
Those features can change the reduced integrand and may introduce multiple contact regions.
The mathematical value of the example is that every primitive, response estimate, and free-boundary constant is visible.

\subsection{Analytic illustration}
\label{sec:numerical}

The formulas of \cref{thm:tax-schedule} can be plotted without solving a numerical optimization problem.
The calculation illustrates the closed-form comparative statics and does not enter any statement or proof.

Take \(\alpha=2\), \(\gamma=1/4\), and compare two laws whose composite means are \(\bar h_\mu=1.35\) and \(\bar h_\nu=1.50\).
Equation \eqref{eq:tax-threshold} gives \(b_\mu=0.35\) and \(b_\nu=0.50\).
Both schedules are then evaluated from \eqref{eq:explicit-tax-schedule}.
The larger composite cost moves the zero-tax region to the right, weakly lowers the schedule everywhere, and lowers it strictly for \(x>0.35\); see \cref{fig:tax-schedules}.

\begin{figure}[H]
\centering
\IfFileExists{figures/tax_schedules.pdf}{%
\includegraphics[width=0.78\textwidth]{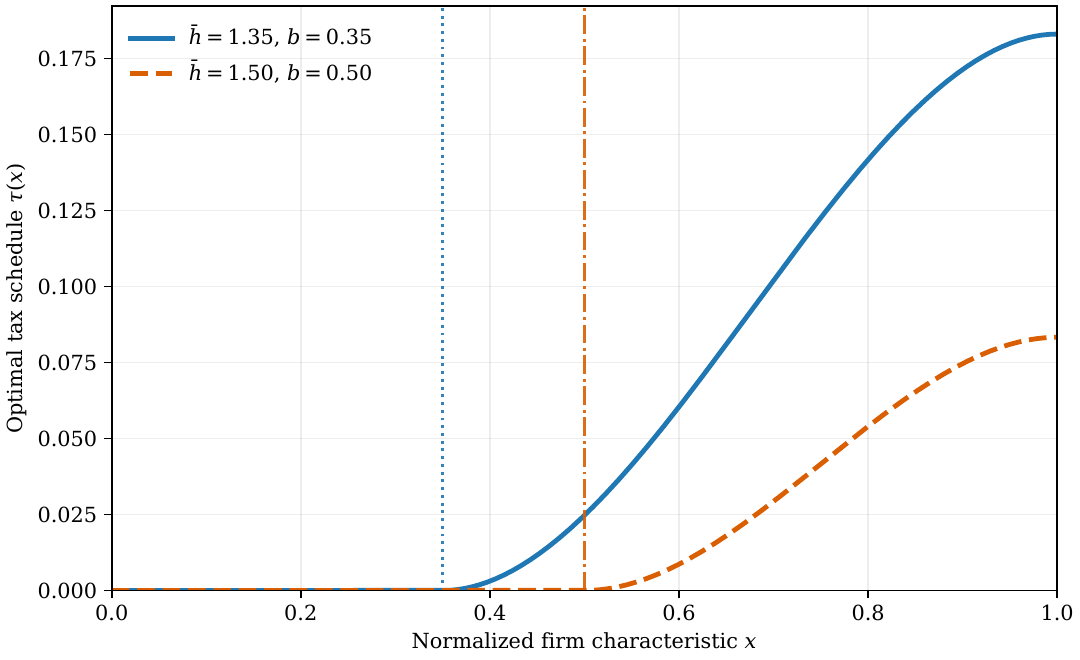}%
}{%
\fbox{\parbox[c][45mm][c]{0.78\textwidth}{\centering Run \texttt{python anc/numerical\_example.py} to generate the analytic illustration.}}%
}
\caption{Closed-form optimal tax schedules for \(\alpha=2\) and \(\gamma=1/4\). The solid blue curve has \(\bar h=1.35\) and threshold \(b=0.35\); the dashed orange curve has \(\bar h=1.50\) and threshold \(b=0.50\). Matching dotted and dash-dotted vertical lines mark the thresholds. Each schedule is zero to the left of its threshold and positive to the right. The parameter values are constructed, and no empirical data are used.}
\label{fig:tax-schedules}
\end{figure}
\FloatBarrier

The script \texttt{anc/numerical\_example.py} generates both PDF and PNG versions of the figure.
It also evaluates the identity
\[
|b_\mu-b_\nu|=\frac2\alpha|\bar h_\mu-\bar h_\nu|,
\]
which here gives \(0.15\) on both sides.
This check verifies an algebraic identity; it is not a simulation-based validation of the theory.

\section{Discussion and outlook}
\label{sec:discussion}

The principal mathematical message is that stability of an optimizer and stability of its active-set boundary are different assertions.
Strong monotonicity controls the optimizer directly and, through a type-Lipschitz reduced gradient, gives a Wasserstein-Lipschitz policy map.
The location of the free boundary additionally requires a quantity that changes sign at a controlled rate.
In the one-dimensional problem, integration of the Euler equation identifies that quantity as the cumulative forcing \(Q_\mu\).
This removes the ambiguity that would arise from trying to use the obstacle solution itself, which has quadratic rather than transversal contact.

The cumulative representation also explains the interval structure.
It is not imposed as an assumption on the contact set.
It follows from the single crossing of \(q_\mu\), the negative total forcing, and the natural Neumann condition at the right endpoint.
The balance equation \(\int_{b_\mu}^1q_\mu=0\) has an intuitive interpretation: the positivity region starts where the total marginal gain remaining to the right is exactly zero.
This identity is more informative than a generic continuous-dependence estimate and is the reason the threshold can move linearly even though the solution grows quadratically away from the obstacle.

The algebraic crossing theorem makes the loss of boundary regularity quantitative.
An order-\(m\) lower bound for the switching function yields an exponent \(1/m\), and the cubic family attains the resulting \(1/3\) modulus uniformly.
The policies in that family remain uniformly close in \(C^1\) while the free boundary moves on the larger cube-root scale.
Thus no refinement of the energy estimate alone can restore a Lipschitz boundary rate.
One must either assume transversality, identify a different nondegenerate switching variable, or accept a weaker modulus tied to the order of contact.

The empirical results deliberately separate two statistical routes.
On a compact interval, the distribution-function formula for \(W_1\) gives a transparent distribution-free \(N^{-1/2}\) bound without density assumptions.
On compact subsets of \(\mathbb R^k\), multiscale transport estimates give ambient-dimension-dependent bounds, which transfer directly to policies and algebraically degenerate thresholds.
By contrast, a transversal threshold is the zero of one smooth scalar estimating equation; its asymptotic linearization therefore yields a root-\(N\) central limit theorem even when the generic \(W_1\) rate is dimension-limited.
There is no contradiction: the first route controls the response to arbitrary law perturbations, whereas the second exploits a local finite-dimensional sufficient direction.

The higher-dimensional statement has a similarly explicit boundary.
It proves the final level-set perturbation step under a stable defining function, but does not derive that function from an arbitrary obstacle problem.
At singular points a single smooth defining function may not exist; even at regular points, its quantitative dependence on the data can require substantial analysis.
Existing obstacle stability theory addresses deeper versions of those questions under its own hypotheses.
The conditional theorem here is useful when a model supplies a switching field directly or when classical regularity and a hodograph construction have already done the difficult geometric work.

These comparisons delimit the contribution.
The paper does not improve the generalized differentiability theory of the full obstacle solution map, the fine structure of singular free boundaries, or the operation complexity of bilevel algorithms.
It provides an end-to-end distributional stability mechanism, with explicit constants, sharp algebraic crossing exponents, and local threshold inference, for a measure-parameterized class of bilevel obstacle problems.

The corporate-tax model shows that the abstract hypotheses are not empty, but it should be interpreted with restraint.
The schedule is optimal only for the stated reduced loss.
The affine local fiscal term and quadratic investment technology were selected to keep the behavioral response and the free boundary analytically visible.
A model with endogenous prices, profit shifting, multiple jurisdictions, private information, or a binding statutory ceiling would alter the forcing and could create additional regimes.
In particular, a binding upper obstacle would produce a second active set and require a two-threshold analysis.

Several extensions appear mathematically plausible.
A semilinear one-dimensional reduced energy would replace the explicit cumulative formula by a nonlinear shooting map; monotonicity and an implicit-function estimate may still yield threshold stability.
A vector-valued policy would require a different notion of active set because the positive cone has corners.
Matching lower bounds for the generic policy and boundary sample complexity would distinguish ambient dimension from intrinsic dimension or finite support.
Finally, models in which the follower response is set-valued would require a measurable selection or an optimistic/pessimistic bilevel convention and could destroy differentiability of the reduced integrand.

These are genuine changes of mathematical regime.
The present paper establishes the stated results under explicit, verifiable assumptions.

\subsection{Conclusion}
\label{sec:conclusion}

We developed a measure-parameterized bilevel obstacle problem from primitive follower optimization through the motion of its active-set boundary.
Uniform strong convexity gives a stable response map; convexity of the reduced integrand gives a unique upper-level policy; and a type-Lipschitz marginal term converts Wasserstein perturbations into energy-space perturbations.
In one dimension, a cumulative-forcing representation yields an interval contact set and an exact scalar equation for the free boundary.
An algebraic crossing of order \(m\) gives a \(W_1^{1/m}\) modulus, with a sharp cubic example; transversality recovers Lipschitz stability.
The deterministic estimates pass to empirical type laws with dimension-dependent rates, while a transversal empirical boundary additionally admits an explicit influence function and a root-\(N\) central limit theorem.

The corporate-tax specialization produces a smooth nonlinear schedule with an endogenous zero-tax region and closed-form Wasserstein comparative statics.
Its role is to verify the mathematical architecture in a transparent economic model, not to make an empirical policy recommendation.
The results therefore identify both a reusable mechanism and its boundary: probability-law stability of a free boundary is available when the law enters through Lipschitz expectations and the active-set interface admits a nondegenerate switching description.

\subsection*{Data and code availability}

This study uses no external data.
The analytic figure and algebraic identity check can be reproduced with the Python script in \texttt{anc/numerical\_example.py}.

\appendix
\section{Auxiliary functional-analytic details}
\label{app:functional}

This appendix records details used implicitly in the main variational arguments.

\begin{lemma}[Closedness of the obstacle cone]
\label{lem:closed-cone}
The set \(\K=\{v\in V:v\geq0\text{ a.e.}\}\) is strongly and weakly closed in \(V\).
\end{lemma}

\begin{proof}
If \(v_n\to v\) in \(V\), then \(v_n\to v\) in \(L^2(\Omega)\).
A subsequence converges almost everywhere, so nonnegativity passes to the limit.
Thus \(\K\) is strongly closed.
Every strongly closed convex subset of a Hilbert space is weakly closed.
\end{proof}

\begin{lemma}[Coercivity with linear lower growth]
\label{lem:coercivity}
Suppose \eqref{eq:poincare} and \eqref{eq:F-lower} hold.
Then \(\J_\mu(u)\to+\infty\) whenever \(u\in\K\) and \(\|u\|_V\to\infty\), uniformly in \(\mu\in\mathfrak M\).
\end{lemma}

\begin{proof}
The estimate in the proof of \cref{thm:wellposed} reads
\[
\J_\mu(u)
\geq
\frac\gamma2\|u\|_V^2
-A\|u\|_V-B,
\]
where \(A=c_0|\Omega|^{1/2}C_P\) and \(B=\|b\|_{L^1}\) are independent of \(\mu\).
The quadratic term dominates as \(\|u\|_V\to\infty\).
\end{proof}

\begin{lemma}[Differentiation of the integral functional]
\label{lem:Gateaux}
Under \eqref{eq:G-growth}, the map
\[
I_\mu(u):=\int_\Omega F_\mu(x,u(x))\dd x
\]
is G\^ateaux differentiable on \(V\), and
\[
I_\mu'(u)v=\int_\Omega G_\mu(x,u(x))v(x)\dd x.
\]
\end{lemma}

\begin{proof}
For \(|t|\leq1\), the mean-value theorem gives
\[
\left|
\frac{F_\mu(x,u+tv)-F_\mu(x,u)}t
\right|
\leq
[a_0(x)+a_1(|u(x)|+|v(x)|)]|v(x)|.
\]
The right side is integrable by H\"older's inequality because \(a_0,u,v\in L^2(\Omega)\).
Pointwise convergence and dominated convergence prove the formula.
\end{proof}

\begin{proposition}[A strongly monotone operator formulation]
\label{prop:operator}
Define \(A_\mu:V\to V^*\) by
\[
\langle A_\mu(u),v\rangle
=\int_\Omega\gamma\nabla u\cdot\nabla v+G_\mu(x,u)v\dd x.
\]
Then
\begin{equation}
\langle A_\mu(u)-A_\mu(v),u-v\rangle
\geq\gamma\|u-v\|_V^2.
\label{eq:operator-strong}
\end{equation}
The upper problem is equivalent to finding \(u_\mu\in\K\) such that
\[
\langle A_\mu(u_\mu),v-u_\mu\rangle\geq0
\quad(v\in\K).
\]
\end{proposition}

\begin{proof}
The gradient contribution in the left side of \eqref{eq:operator-strong} is \(\gamma\|u-v\|_V^2\).
The remaining integral is nonnegative by \eqref{eq:G-monotone}.
The variational inequality is \eqref{eq:VI} in operator notation.
\end{proof}

The estimate \eqref{eq:operator-strong} explains why a positive quadratic term \(\lambda\int u^2\) was unnecessary.
Adding one would strengthen coercivity but would also obscure the elementary cumulative representation used for the free-boundary theorem.

\section{Differentiation through the response map}
\label{app:response-checks}

The envelope hypothesis \eqref{eq:g-envelope} follows from the stronger componentwise condition that, for every bounded \(z\)-interval, there is an integrable function \(M(x,\theta)\) such that
\[
|\partial_zL(x,\theta,z,R)|
+|D_yL(x,\theta,z,R)||D_zR(x,\theta,z)|
\leq M(x,\theta).
\]
Under the differentiability assumptions of \cref{prop:primitive-to-reduced}, the chain rule gives \eqref{eq:primitive-g}, and dominated convergence permits
\[
\partial_zF_\mu(x,z)
=\int_\Theta\partial_zf(x,\theta,z)\,\mu(\dd\theta).
\]
In the corporate-tax model the calculation is even simpler: \(f\) is affine in \(z\), all functions are bounded on the compact type space, and the derivative follows by direct algebra.

\begin{remark}[Alternative boundary conditions]
The anchoring condition at \(x=0\) serves two purposes: it supplies Poincar\'e coercivity and fixes the additive constant in the one-dimensional problem.
With pure Neumann conditions, the linear functional must satisfy an additional compatibility or coercivity condition in the constant direction.
With Dirichlet conditions at both endpoints, the switching formula changes because both endpoint slopes are constrained.
The one-sided Dirichlet and one-sided natural boundary conditions are therefore part of the model, not a disposable notation choice.
\end{remark}

\section{Explicit calculations for affine forcing}
\label{app:affine}

Let \(q_h(x)=\alpha x-h\) with \(h\in(\alpha/2,\alpha)\).
Then
\[
Q_h(x)
=\int_x^1(\alpha s-h)\dd s
=(1-x)\left(\frac\alpha2(1+x)-h\right).
\]
The nontrivial root satisfies
\[
\frac\alpha2(1+b_h)-h=0,
\qquad
b_h=\frac{2h}{\alpha}-1.
\]
Since \(h<\alpha\),
\[
b_h<\frac h\alpha,
\]
so the free boundary lies strictly to the left of the point where the local forcing changes sign.
This gap reflects aggregation of the positive forcing remaining to the right.

Using \(h=\alpha(1+b_h)/2\) gives
\[
Q_h(x)=\frac\alpha2(1-x)(x-b_h).
\]
For \(s=x-b_h\) and \(L=1-b_h\),
\begin{align*}
u_h(x)
&=\frac\alpha{2\gamma}\int_{b_h}^x(1-t)(t-b_h)\dd t\\
&=\frac\alpha{2\gamma}\int_0^s(L-r)r\dd r\\
&=\frac\alpha{2\gamma}\left(\frac{Ls^2}{2}-\frac{s^3}{3}\right)\\
&=\frac\alpha{12\gamma}s^2(3L-2s).
\end{align*}
At \(x=1\), \(s=L\), so \(u_h(1)=\alpha L^3/(12\gamma)\).

The multiplier on the contact interval is
\[
\xi_h(x)=-q_h(x)=h-\alpha x>0
\quad(0<x<b_h),
\]
and vanishes from the complementarity equation on \((b_h,1)\).
At the free boundary,
\[
-q_h(b_h)
=h-\alpha\left(\frac{2h}{\alpha}-1\right)
=\alpha-h>0.
\]
Consequently
\[
u_h(b_h+s)
=\frac{\alpha-h}{2\gamma}s^2+O(s^3),
\]
which agrees with \cref{cor:quadratic-growth}.

\section{The degenerate family in complementarity form}
\label{app:degenerate}

For the family in \cref{ex:degenerate}, define
\[
u_a(x)=\frac1\gamma\int_0^x[Q_a(t)]_+\dd t.
\]
Because \(Q_a<0\) before \(b_a\) and \(Q_a>0\) after it,
\[
u_a=0\text{ on }[0,b_a],
\qquad
u_a'=\frac{Q_a}{\gamma}>0\text{ on }(b_a,1).
\]
On the positivity interval, \(-\gamma u_a''-q_a=0\).
On the contact interval, \(u_a''=0\) and
\[
-q_a=Q_a'\geq0.
\]
Thus the family is not merely an abstract collection of level sets; every member is the unique solution of a legitimate obstacle variational inequality.
Furthermore,
\[
Q_a-Q_0=-a(1-x),
\qquad
q_a-q_0=-a.
\]
The positive-part contraction yields
\[
\|u_a'-u_0'\|_\infty\leq\frac a\gamma,
\qquad
\|u_a-u_0\|_\infty\leq\frac a\gamma,
\]
while the boundary displacement is \(a^{1/3}\).
This direct comparison separates analytic stability of the state from geometric stability of the active set.

\bibliographystyle{plainnat}
\bibliography{references}

\end{document}